\documentclass[11pt]{amsart}

\usepackage{amscd,amssymb,amsmath,graphicx,verbatim}
\usepackage{wasysym}
\usepackage{mathrsfs}
\usepackage{hyperref}
\usepackage{lscape} 
\usepackage{fullpage} 
\usepackage{tikz}
\usepackage{dsfont}
\usetikzlibrary{shapes,arrows}
\usepackage[all]{xy}
\newtheorem{theorem}{Theorem}[section]
\newtheorem{lemma}[theorem]{Lemma}
\newtheorem{corollary}[theorem]{Corollary}
\newtheorem{proposition}[theorem]{Proposition}

\theoremstyle{definition}
\newtheorem{definition}[theorem]{Definition}

\newtheorem{example}[theorem]{Example}
\newtheorem{examples}[theorem]{Examples}

\theoremstyle{remark}
\newtheorem{remark}[theorem]{Remark}

 \DeclareMathOperator{\Hom}{Hom}

\DeclareMathOperator{\Ext}{Ext} \DeclareMathOperator{\End}{End}
 
\DeclareMathOperator{\fdmod}{mod}

\DeclareMathOperator*{\lmod}{\!-\mathrm{mod}}
\DeclareMathOperator*{\Modr}{\mathrm{Mod}-\!}
\DeclareMathOperator*{\modr}{\mathrm{mod}-\!}

\DeclareMathOperator{\Mcolim}{Mcolim}
\DeclareMathOperator{\hocolim}{\mathrm{hocolim}}
\DeclareMathOperator{\holim}{\mathrm{holim}}
\DeclareMathOperator{\RHom}{\mathbf{R}Hom}
\newcommand{\Dcal}{\ensuremath{\mathcal{D}}}
\newcommand{\Xcal}{\ensuremath{\mathcal{X}}}
\newcommand{\Ycal}{\ensuremath{\mathcal{Y}}}
\newcommand{\Tcal}{\ensuremath{\mathcal{T}}}
\newcommand{\T}{\ensuremath{\mathcal{T}}}
\newcommand{\Gcal}{\ensuremath{\mathcal{G}}}
\newcommand{\Fcal}{\ensuremath{\mathcal{F}}}
\newcommand{\Ccal}{\ensuremath{\mathcal{C}}}

\newcommand{\Bcal}{\ensuremath{\mathcal{B}}}
\newcommand{\Ab}{\mathsf{Ab}}

\newcommand{\Acal}{\ensuremath{\mathcal{A}}}
\newcommand{\Wcal}{\ensuremath{\mathcal{W}}}
\newcommand{\Hcal}{\ensuremath{\mathcal{H}}}
\newcommand{\Scal}{\ensuremath{\mathcal{S}}}
\newcommand{\Pcal}{\ensuremath{\mathcal{P}}}
\newcommand{\Ucal}{\ensuremath{\mathcal{U}}}
\newcommand{\Vcal}{\ensuremath{\mathcal{V}}}
\newcommand{\Zcal}{\ensuremath{\mathcal{Z}}}

\newcommand{\Add}{\mathrm{Add}}
\newcommand{\add}{\mathrm{add}}
\newcommand{\Prod}{\mathrm{Prod}}

\newcommand{\Zbb}{\mathbb{Z}}
\newcommand{\p}{\mathfrak{p}}

\newcommand{\Spec}{\mathrm{Spec}}
\newcommand{\Supp}{\mathrm{Supp}}

\DeclareMathOperator{\Ker}{Ker}

\numberwithin{equation}{section}

\begin{document}

\title{Silting objects}
\author{Lidia Angeleri H\"ugel}
\address{Lidia Angeleri H\"ugel, Dipartimento di Informatica - Settore di Matematica, Universit\`a degli Studi di Verona, Strada le Grazie 15 - Ca' Vignal, I-37134 Verona, Italy} \email{lidia.angeleri@univr.it}
\maketitle


\begin{abstract}
We give an overview of recent developments in silting theory.  After an introduction on torsion pairs in triangulated categories, we discuss and compare different notions of silting and explain the interplay with t-structures and co-t-structures.  We then focus on silting and cosilting objects in a triangulated category with coproducts and study the case of the unbounded derived category of a ring. We close the survey with some classification results over commutative noetherian rings and silting-discrete algebras.\end{abstract}
{\small\tableofcontents}


\section{Introduction} 



Silting theory is a topic originating from representation theory of algebras with far-reaching applications. As an extension of tilting theory, it encompasses methods for studying derived equivalences which are widely employed in many areas of research, ranging from algebraic geometry and topology to Lie theory and cluster theory.
Tilting theory starts from the problem of determining when two rings have the same representation theory. Morita proved that two rings  have equivalent module categories if and only if one  arises as the endomorphism ring of a special module, called  progenerator, over the other. Rickard carried this phenomenon to the derived setting, proving that  two rings have equivalent derived module categories if and only if one  arises as the endomorphism ring of a special complex, called  tilting complex, over the other. These techniques were later extended by Happel, Reiten and Smal\o\  to  abelian categories and revealed the existence of derived equivalences between categories of representations and categories coming from other areas of mathematics, most notably categories of geometric origin. 

Silting theory goes a step further and considers a larger collection of abelian subcategories of the derived category, including some that are not derived equivalent to the original category.  Silting objects in triangulated categories determine torsion pairs, that is, decompositions into orthogonal subcategories from which the original category can be reconstructed by taking extensions. They are thus intimately related with approximation theory and localisation theory. 

Another important motivation for the recent interest in silting theory comes from cluster theory. The categorification of cluster algebras via representation theory allows to interpret clusters as silting objects and cluster mutation as mutation of silting objects. Silting mutation
is an operation that produces a new silting object from a given one by replacing a summand. It
can be viewed as a completion of tilting mutation, since it overcomes the problem that mutations of tilting objects are not always possible.
In fact, in many interesting situations iterated silting mutations even act transitively on the set of silting objects.

The notion of silting appeared for the first time in 1988 in a paper of Keller and Vossieck \cite{KV} devoted to the classification of certain t-structures. Silting complexes were introduced as a generalisation of the fundamental notion of a tilting complex.
In the following years, the topic was studied only sporadically, e.g.~in \cite{HKM,AST}, until 2012, when it was rediscovered by Aihara and Iyama in connection with cluster mutation \cite{AI}. In a parallel development, silting objects also reemerged in the context of approximation theory for triangulated categories. Results of Mendoza, S\'aenz, Santiago and Souto Salorio \cite{MSSS}, Wei \cite{Wei}, Keller and Nicol\'as \cite{KN}, and Koenig and Yang \cite{KY} establish important correspondences relating silting complexes, t-structures and co-t-structures. A survey on these connections and on the interaction with cluster theory can be found in \cite{BY}.

An important role is played by the case of a two-term silting complex, which had already been investigated by Hoshino, Kato and Miyachi in \cite{HKM}. The zero cohomologies of such complexes form an interesting class of modules, first studied systematically by Adachi, Iyama and Reiten in \cite{AIR}, which parametrise torsion pairs in module categories and are also related to ring theoretic localisation \cite{AMV1, AMV2}. Ongoing work on these modules concerns several representation theoretic and combinatorial aspects, see e.g.~\cite{BZ1,BZ2,DIJ,IRTT,AMV4,BMmodule,BuM2}. Some of these topics are treated in the surveys \cite{IR} and \cite{abundance}.

In this account, we will follow a more abstract approach centred around
the notion  of a silting object in a triangulated category with coproducts, introduced independently by Psaroudakis and Vit\'oria \cite{PV}, and by Nicol\'as, Saor\'in and Zvonareva in \cite{NSZ}. We will focus on the interplay with t-structures and co-t-structures. It will turn out that silting objects correspond bijectively to certain triples, called TTF triples, formed by a t-structure with a left adjacent co-t-structure. Although no compactness is required in our definition of silting, we will see that the related TTF triples are often Òof finite typeÓ, in the sense that they are determined by a set of compact objects, cf.~\cite{MV}.

We will also consider the dual notion of a cosilting object. Cosilting objects are often pure-injective, in which case they give rise to TTF triples formed by a t-structure with a right adjacent co-t-structure. Pure-injectivity entails that the t-structure has a number of nice properties. For example, its heart, which is known to be an abelian category, is even a Grothendieck category in this case.
In particular, it turns out that every nondegenerate compactly generated t-structure occurs in such a TTF-triple and thus shares these properties, cf.~\cite{AMV3,Bondarko2016,SSV,L}.

The reader familiar with the theory of tilting modules will have recognised some well-known features. Every tilting module  of projective dimension at most one gives rise to a triple formed by a torsion pair with a left adjacent cotorsion pair, and cotilting modules  of injective dimension at most one have the dual property. Moreover, the triples corresponding to tilting modules are Òof finite typeÓ, that is, they are determined by a set of finitely presented modules of projective dimension at most one, and cotilting modules are pure-injective.

In view of these analogies, it is not too surprising that classification problems for silting or cosilting objects can be approached following the same strategy adopted for modules. Indeed, over any ring $R$, there is a duality mapping silting complexes of right modules to cosilting complexes of left modules which are Òof cofinite typeÓ, in the sense that the corresponding TTF triple is determined by a set of compact objects.
In general, not all cosilting complexes are obtained in this way. But if the ring $R$ is commutative noetherian, one can exploit known classification results for compactly generated t-structures to prove that all cosilting complexes are of cofinite type, and to obtain a parametrisation of silting complexes and cosilting complexes by certain descending sequences of subsets of the prime spectrum of $R$. Notice, however, that this classification covers only bounded complexes, while the unbounded derived category of $R$ contains also silting or cosilting objects that are unbounded (Example~\ref{unbounded}).

Further classification results are obtained for silting-discrete algebras, which are characterised by the property that all silting complexes are compact, up to equivalence. This class of algebras includes  discrete-derived algebras of finite global dimension \cite{BPP2}, preprojective algebras of Dynkin type \cite{AM}, and certain symmetric algebras \cite{Adachi, Eisele, AAChan}.

\smallskip

The paper is organised as follows. After some preliminaries in Section 2, we devote Section 3 to the notion of a torsion pair in a triangulated category. In particular, we collect some results on the existence of approximations in triangulated categories which can be used to produce t-structures and co-t-structures. In Section 4 we introduce the notion of silting. We discuss and compare silting subcategories in arbitrary triangulated categories in the sense of \cite{AI} and silting objects in triangulated categories with coproducts in the sense of \cite{PV,NSZ}, and we explain the connections with torsion pairs. Section 5 deals with the case of a derived module category. We focus on a special class of silting objects: (bounded) silting complexes. We present bijections with t-structures and co-t-structures. Moreover, we devote some attention to the special case of two-term silting complexes and the related silting and support $\tau$-tilting modules. We also review some results on endomorphism rings of silting complexes and on generalisations of the Tilting Theorem to this setting. 
 Secton 6 starts with a brief reminder on the notion of purity in triangulated categories and then treats the dual notions of a cosilting object and a cosilting complex. The classification results mentioned above are presented in Section 7.

\smallskip


\medskip
\noindent {\bf Acknowledgements.}  
 {The author acknowledges funding from Istituto Nazionale di Alta Matematica INdAM-GNSAGA  and from the Project ``Ricerca di Base 2015'' of the University of Verona.}


\section{Preliminaries}
\subsection{Notation}\label{notation} Throughout, we denote by $\T$ a triangulated category with shift functor $[1]$. All subcategories considered in this note are strict and full. 

Given a triangle $X\rightarrow Y\rightarrow Z\rightarrow X[1]$ in $\T$, the  object
 $Y$  is said to be an \emph{extension} of $X$ and $Z$. If $\Xcal$ and $\Zcal$ are two classes of objects in $\T$, we denote by $\Xcal\ast\Zcal$ the class of all objects $Y$ that fit in a triangle as above with $X\in\Xcal$ and $Z\in\Zcal$.
 For a set of integers $I$ (which is often expressed by symbols such as $>n$, $<n$, $\geq n$, $\leq n$, $\neq n$, or just $n$, with the obvious associated meaning) 
 we define the following orthogonal classes
$${}^{\perp_I}\Xcal:=\{Y\in \T:\Hom_\T(Y,X[i])=0, \text{ for all } X\in\Xcal \text{ and } i\in I\}$$ $${\Xcal}^{\perp_I}:=\{Y\in \T:\Hom_\T(X,Y[i])=0, \text{ for all } X\in\Xcal \text{ and }  i\in I\}.$$
Furthermore, 
 we  denote by $\add(\Xcal)$ the smallest subcategory of $\T$ containing $\Xcal$ and closed under  finite coproducts and summands.
If $\T$ has coproducts (respectively, products), we denote by $\Add(\Xcal)$ (respectively, $\Prod(\Xcal)$)
  the smallest subcategory of $\T$ containing $\Xcal$ and closed under coproducts (respectively,  products) and summands. 
 If $\Xcal$ consists of a single object $M$, we write ${}^{\perp_I}M$, $M^{\perp_I}$,
 $\add(M)$,  $\Add(M)$, and $\Prod(M)$.

Recall that  an additive functor from $\Tcal$ to an abelian category $\Acal$ is said to be \emph{cohomological} if it takes triangles in $\T$ to long exact sequences in $\Acal$. Moreover, a \emph{triangulated subcategory} of $\T$ is a subcategory which is closed under extensions and  shifts; if it is also closed under direct summands, then it is called 
\emph{thick}. Given a class  of objects $\Pcal$ in $\Tcal$, we denote by $\mathsf{thick}(\Pcal)$ the smallest thick subcategory of $\T$ containing $\Pcal$. Finally, we  say that a class of objects $\Pcal$  \emph{generates}  $\T$  if $\Pcal^{\perp_{\mathbb Z}}=0$, and when  $\Pcal$ consists of a single object $M$, we call $M$ a \emph{generator} of $\T$. Of course, $\Pcal$ generates $\T$ whenever $\mathsf{thick}(\Pcal)=\T$.

\subsection{Triangulated categories with coproducts}
Let now $\T$ be  a triangulated category with arbitrary (set-indexed) coproducts.
Then $\T$  is \emph{idempotent complete}, i.e.~
if $X$ is an object of $\T$ with an  endomorphism $e : X \to X$ such that $e^2 = e$, then $e$  factors  through an object  $Y$ in $\T$ via morphisms $f:X\to Y$ and $g:Y\to X$ such that $gf = e$ and $fg = 1_Y$
 (see \cite[Proposition 1.6.8]{Neeman}).
 
Moreover, given a  sequence of morphisms
$X_0 \xrightarrow{f_0} X_1 \xrightarrow{f_1} X_2 \xrightarrow{f_2} \cdots$
  in $\Tcal$,  we can form the \emph{Milnor colimit} $\Mcolim_{n \geq 0}X_n$, which  is defined as the cone of the morphism
$\coprod_{n \geq 0} X_n \xrightarrow{1-f} \coprod_{n \geq 0}X_n$
given by $f=\coprod_{n \geq 0}f_n$.
 
 \smallskip
 
 Recall that an object $X\in\T$ is said to be \emph{compact} if the functor $\Hom_T(X,-)$ commutes with coproducts. If the subcategory of compact objects, denoted by $\T^c$, is skeletally small
and generates $\T$, then  $\T$ is said  to be 
\emph{compactly generated}.   It is  well-known (see \cite[Proposition 8.4.6 and Theorem 8.3.3]{Neeman}) that $\T$ then  also admits products.


\subsection{Derived categories.}
Given an abelian category $\Acal$, we denote by $\mathsf{D}(\Acal)$ its unbounded derived category.
For a ring $A$, we denote by $\Modr A$   the category of all right $A$-modules and write $\mathsf{D}(A)=\mathsf{D}(\Modr A)$. The subcategories of injective, of projective, and of finitely generated projective $A$-modules are denoted, respectively, by $\mathrm{Inj}(A)$, by $\mathrm{Proj}(A)$, and  $\mathrm{proj}(A)$. Their bounded homotopy categories are denoted by $\mathsf{K}^b(\mathrm{Inj(A)})$, by $\mathsf{K}^b(\mathrm{Proj}(A))$, and $\mathsf{K}^b(\mathrm{proj}(A))$, respectively. It is well known that $\mathsf{K}^b(\mathrm{proj}(A))=\mathsf{D}(A)^c$ and $\mathsf{D}(A)$ is compactly generated.
The  category of all finitely presented right $A$-modules is denoted by  $\modr A$, and  $\mathsf{D}^b(\modr A)$ is its bounded derived category.

Finally, a triangulated category is said to be \emph{algebraic} if it can be constructed as the stable category of a Frobenius exact category. Note that algebraic and compactly generated triangulated categories are essentially derived categories of small differential graded categories, see \cite{Keller}.


\section{Torsion pairs}

Torsion pairs have their origin in localisation theory and are widely used in algebra, geometry, and topology. A torsion pair  (or torsion theory) provides a decomposition of a category in smaller parts that are still big enough to allow for reconstruction of the whole category.

Silting theory is intimately related with the theory of  torsion pairs. In this section, we review the relevant notions and collect some useful existence results. We then focus on TTF triples, that is, triples of classes formed by two adjacent torsion pairs. We will see in the next sections that silting and cosilting objects correspond bijectively to certain TTF triples in the ambient triangulated category.

 \subsection{Basic terminology} 
A pair of subcategories    $(\Xcal,\Ycal)$ in an abelian category $\Acal$ is a {\it torsion pair} if
 $\Hom_A(\Xcal,\Ycal)=0$,
and every object in $\Acal$ is an extension of an object in    $\Xcal$ by an object in $\Ycal$.
Then $\Xcal$ is called the {\it torsion class} and $\Ycal$  the {\it torsion-free class} of the torsion pair. Torsion classes are always closed under extensions and quotients, torsion-free classes are closed under extensions and subobjects. A torsion pair   $(\Xcal,\Ycal)$ is \emph{hereditary} if the torsion class $\Xcal$ is also closed under subobjects.

\smallskip

Let us now consider the notion of a torsion pair in a triangulated category $\T$  and specialise it to the notions of a t-structure  and a co-t-structure introduced in \cite{BBD} and \cite{Bondarko2010,Pauk}, respectively.  
\begin{definition}\label{def tp}
A pair of subcategories $\mathfrak t=(\Ucal,\Vcal)$ in $\T$ is said to be a \emph{torsion pair} if
\begin{enumerate}
\item $\Ucal$ and $\Vcal$ are closed under summands;
\item $\Hom_\T(\Ucal,\Vcal)=0$;
\item  $\T=\Ucal\ast\Vcal$.
\end{enumerate}
The class $\Ucal$ is then called the \emph{aisle}, the class $\Vcal$ the \emph{coaisle} of $\mathfrak t$, and the torsion pair $\mathfrak t$ is said to be
\begin{itemize}
\item a \emph{t-structure} if $\Ucal[1]\subseteq \Ucal$, in which case we say that $\Hcal_{\mathfrak t}=\Ucal\cap \Vcal[1]$ is the \emph{heart} of $(\Ucal,\Vcal)$;
\item a \emph{co-t-structure} if $\Ucal[-1]\subseteq \Ucal$, in which case we say that $\Ucal[1]\cap \Vcal$ is the \emph{coheart} of $(\Ucal,\Vcal)$;
\item $\emph{nondegenerate}$ if $\bigcap_{n\in\mathbb{Z}}\Ucal[n]=0=\bigcap_{n\in\mathbb{Z}}\Vcal[n]$;
\item $\emph{bounded}$ if $\bigcup_{n\in\Zbb}\Ucal[n]=\T=\bigcup_{n\in\Zbb}\Vcal[n]$, in which case it is also nondegenerate;
\item \emph{generated by a subcategory $\Scal$} of $\T$ if $(\Ucal,\Vcal)=({}^{\perp_0}(\Scal^{\perp_0}),\Scal^{\perp_0})$; 
\item \emph{compactly generated} if it is generated  by a subcategory of  $\T^c$.
\end{itemize}
\end{definition}

\begin{remark} Beware that the use of terminology and notation may vary in the literature.
For example,   torsion pairs as above are  termed ``complete Hom-orthogonal pairs'' in \cite{StPo}, while  torsion pairs according to \cite[Definition I.2.1]{BR} are in fact t-structures. Co-t-structures are called weight structures in \cite{Bondarko2010}.
Moreover, both t-structures and   co-t-structures are often written as a pair $(\Ucal, \Vcal)$ with a non-empty intersection (the heart or the coheart, respectively), differing  by a shift from the notation above.
\end{remark}

Torsion pairs give rise to approximations. Indeed, condition (3) in Definition~\ref{def tp} states that every object $X\in \T$ admits  a triangle 
\begin{equation}\label{apprtriangle}
\xymatrix{U\ar[r]^{\ f} & X\ar[r]^{\ \ g \ \ \ }& V\ar[r]& U[1]}
\end{equation}
with $U\in\Ucal$ and $V\in\Vcal$. Combined with condition (2), we see that the morphism $f$ is then a $\Ucal$-\emph{precover} (or a \emph{right $\Ucal$-approximation}) of $X$, i.~e.~every morphism $f':U'\rightarrow X$ with $U'\in\Ucal$ factors through $f$. Dually, $g$ is  a $\Vcal$-\emph{preenvelope}  (or a \emph{left $\Vcal$-approximation}) of $X$, i.~e.~every morphism $g':X\rightarrow V'$ with $V'\in\Vcal$ factors through $g$.

Indeed, a stronger property  is verified when $\mathfrak{t}=(\Ucal,\Vcal)$ is a t-structure. In this case, the condition $\Ucal[1]\subseteq \Ucal$  implies
that $g'$ factors uniquely through $g$, and $f'$ factors uniquely through $f$.
In other words, the approximation triangle (\ref{apprtriangle}) can  be expressed functorially as 
\begin{equation}\label{apprtriangle2}
\xymatrix{u(X)\ar[r]^{\ f} & X\ar[r]^{g \ \ \ }& v(X)\ar[r]& u(X)[1]}\end{equation}
where  $u:\T\longrightarrow \Ucal$ is the right adjoint of the inclusion of $\Ucal$ in $\Tcal$ and $v:\T\longrightarrow \Vcal$ is the left adjoint of the inclusion of $\Vcal$ in $\Tcal$. The existence of one of these adjoints, usually called \emph{truncation functors}, is in fact equivalent to the fact that $(\Ucal,\Vcal)$ is a t-structure, cf.~\cite[Proposition 1.1]{KV}. Observe that the maps $f$ and $g$ in the triangle are, respectively, the counit and unit map of the relevant adjunction.
In particular, it follows that if $f=0$ (respectively, $g=0$), then $u(X)=0$ (respectively, $v(X)=0$).

Moreover, when $\mathfrak{t}=(\Ucal,\Vcal)$ is a t-structure,
we know from \cite{BBD} that the heart $\Hcal_\mathfrak{t}$  is an abelian category with the exact structure induced by the triangles of $\T$ lying in $\Hcal_\mathfrak{t}$.
Then $u$ and $v$  can be used to construct   a cohomological functor $H^0_\mathfrak{t}\colon \T\longrightarrow \mathcal{H}_\mathfrak{t}$  defined by
\[
{H}^0_{\mathfrak{t}}(X):
=u(v(X)[1]). 
\]
If $(\Ucal,\Vcal)$ is a nondegenerate t-structure, then ${H}^0_{\mathfrak{t}}$ detects the aisle and the coaisle:  $\Ucal$ consists of the objects $X$ with ${H}^0_{\mathfrak{t}}(X[n])=0$ for all $n>0$, and  $\Vcal$ of those with ${H}^0_{\mathfrak{t}}(X[n])=0$ for all $n\le 0$.

\medskip

\begin{examples}\label{standard}
Let $A$ be a ring and $\Tcal=\mathsf D(A)$. 

(1) 
For $n\in\mathbb Z$ let 
\[
\mathsf D^{\le n}:=\{X\in \mathsf D(A)\mid H^i(X)=0\ \forall \ i>n \}
\]
\[
\mathsf D^{> n}:=\{X\in \mathsf D(A)\mid H^i(X)=0\ \forall \ i\le n \} 
\]
 denote the class of all complexes with cohomologies concentrated in degrees $\le n$ and $>n$, respectively. 
Then the pair $(\mathsf D^{\le n},\mathsf D^{>n})$ is a  $t$-structure in $\T$. It  is nondegenerate, but not bounded.
 In case $n=0$, it is called the 
\emph{standard $t$-structure}, and its heart equals $\Modr A$. 

\smallskip

(2) We review a construction due to Happel, Reiten and Smal\o~\cite{HRS}. Given a torsion  pair     $(\Xcal,\Ycal)$ in the module category $\Modr A$, we set 
\[\Vcal:=\{X\in \Dcal^{\le 0}\mid H^0(X)\in\Xcal\}\]
\[\Wcal:=\{X\in\Dcal^{\ge 0}\mid H^0(X)\in \Ycal\}\]
Then 
$(\Vcal, \Wcal)$ 
is a 
$t$-structure in $\T$, called the {\em HRS-tilt} of the torsion pair $(\Xcal,\Ycal)$.
\end{examples}
 
Assume that $\T$ is a triangulated category with (co)products. Then hearts of t-structures also have (co)products: the (co)product of a family of objects in the heart is obtained by applying the functor $H^0_\mathfrak{t}$ to the corresponding (co)product of the same family in $\T$, see \cite{PaSa}. This (co)product, however,   may differ from the (co)product formed in $\T$, and the cohomological functor associated to the t-structure need not commute with  (co)products in $\T$. 
A good behaviour is ensured by the following conditions.
\begin{definition} A t-structure $\mathfrak{t}=(\Ucal,\Vcal)$ in a  triangulated category with coproducts is said to be \emph{smashing}  if the coaisle $\Vcal$ is closed under coproducts. 
 \emph{Cosmashing} t-structures are defined dually.
\end{definition}
By \cite[Lemma 3.3]{AMV3},  a nondegenerate t-structure
 $\mathfrak{t}$ in $\T$ is smashing (respectively, cosmashing) if and only if
  the functor $H^0_\mathfrak{t}$ preserves $\Tcal$-coproducts (respectively, $\Tcal$-products).
  
\smallskip

We will need some further terminology. 

 \begin{definition} A subcategory $\Ucal$ of $\T$ is said to be \emph{suspended} (respectively, \emph{cosuspended}) if it is closed under extensions and positive (respectively, negative) shifts. 
 \end{definition}
 
 For example, a torsion pair $(\Ucal,\Vcal)$ is a t-structure if and only if $\Ucal$ is suspended (or, equivalently, $\Vcal$ is cosuspended). A dual statement holds true for co-t-structures.

\begin{lemma}
\label{swindle}
Let  $\T$ be a triangulated category with coproducts. A  suspended subcategory $\Xcal$  of $\Tcal$ is closed under direct summands provided that the countable coproduct $X^{(\mathbb N)}$ of any object $X$ in $\Xcal$   belongs to $\Xcal$ as well. 
\end{lemma} 
\begin{proof}
If an object $X$ in $\Xcal$ has a decomposition $X=Y\oplus Z$, the direct summand $Y$ can be viewed as  Milnor colimit of a sequence of the form $X\xrightarrow{e} X \xrightarrow{e} X\xrightarrow{e}\cdots$, cf.~\cite[proof of Proposition 1.6.8]{Neeman}. Hence $Y$ is the 
cone of a morphism
$ X^{(\mathbb N)}\rightarrow X^{(\mathbb N)}$, which entails by assumption that $Y$ lies in $\Xcal$. 
\end{proof}

Given a set of objects $\Xcal$ in $\Tcal$, we denote by  $\mathsf{Susp}(\Xcal)$  the smallest suspended subcategory of $\Tcal$ containing $\Xcal$ and closed under all existing coproducts. By the Lemma above, in a triangulated category with coproducts, $\mathsf{Susp}(\Xcal)$ is automatically closed under summands.

\subsection{Some approximation theory}
In order to construct torsion pairs, we will need some results on the existence of precovers and preenvelopes. Recall that  an $\Xcal$-\emph{cover} of an object $M$ in $\Tcal$ is an  $\Xcal$-precover $g:X\to M$ which  is right minimal, i.e.~every endomorphism $s\in\End_\T(X)$ satisfying $gs=g$ is an automorphism. We will say that a subcategory $\Xcal$ of $\Tcal$ is \emph{(pre)covering}  if every object of $\T$ admits an $\Xcal$-(pre)cover.  Envelopes and (pre)enveloping classes are defined dually.

First of all, observe that every object $T$ in a triangulated category with coproducts $\T$ gives rise to a precovering  subcategory $\Add(T)$. Indeed, for any object $X$ in $\T$ we obtain an $\Add(T)$-precover by taking the \emph{universal morphism} $f:T^{(I)}\rightarrow X$, where $I=\Hom_\T(T,X)$ and $f$ is the codiagonal map given by all morphisms from $T$ to $X$. Dually, every object $T$ in a triangulated category with products gives rise to a preenveloping  subcategory $\Prod(T)$.

\smallskip

Subcategories of $\T$ which are even covering are often aisles of torsion pairs, thanks to the following  triangulated analog of a result for abelian categories known as Wakamatsu's Lemma, see also \cite[Proposition 2.3]{IYo}.

\begin{proposition}  \cite[Lemma 2.1]{J}
Let $\T$ be a triangulated category, and let $\Vcal$ be a subcategory  of $\T$  which  is closed under extensions and direct summands. 
Then the cone of every $\Vcal$-cover belongs to $\Vcal ^{\perp_0}$.
\end{proposition}

Next, we present a result by Mendoza, S\'aenz, Santiago, and Souto Salorio which provides a tool for the construction of co-t-structures. It is inspired by an analogous result for module categories due to Auslander and Buchweitz. The idea is to start with a precovering or preenveloping subcategory and to perform  iterated  push-out or pull-back constructions.

We start with the following easy observation which will be useful in the sequel.
\begin{lemma} Let  $\Xcal$ be a class of objects in $\Tcal$, and $Z$ be  an object in $\T$. The following statements are equivalent.
\begin{enumerate}
\item There are a positive integer $n$ and  a finite sequence of triangles 
$$Z_i[-1]\longrightarrow Z_{i+1}\longrightarrow X_i\longrightarrow Z_{i},\quad 0\le i<n,$$
where $Z=Z_0$, and the objects  $X_0,\ldots, X_{n-1}, Z_n$  lie in  $\Xcal$.
\item  $Z$ belongs to the  iterated extension   $ \Xcal\ast\Xcal[1]\ast\ldots\ast\Xcal[n]$ for some $n\ge 0$.
\end{enumerate} 
\end{lemma}
Under the conditions of the Lemma above,  the object $Z$ is said to be \emph{finitely resolved by objects of} $\Xcal$.  
The objects in  iterated extensions of the form $\Xcal[-n]\ast\ldots\ast\Xcal[-1]\ast\Xcal$ have a dual characterisation and are said to be \emph{finitely coresolved by objects of} $\Xcal$.

\begin{theorem}
\cite[Proposition 2.10]{MSSS1}
\cite[Theorem 3.11]{MSSS}\label{mombo}
Let $\T$ be a triangulated category, and let   $(\Xcal,\omega)$ be a pair of classes of objects in $\T$ which are closed under direct summands, and such that
\begin{enumerate}
\item[(i)] $\Xcal$  is a cosuspended subcategory of $\Tcal$,
\item[(ii)] $\omega\subset\Xcal\cap\Xcal^{\perp_{>0}}$,
\item[(iii)] for every  $X\in \Xcal$ there is a triangle $X\longrightarrow C\longrightarrow X'\longrightarrow X[1]$ with $C\in\omega$ and $X'\in\Xcal$.
\end{enumerate}
Then the class 
$\widehat{\Xcal}$ of all objects of $\T$ which are finitely resolved by objects of $\Xcal$  coincides with $\mathsf{thick}(\Xcal)$, and it admits  a co-t-structure $(\Xcal[-1],\widehat{\omega})$ with coheart $\omega$.
\end{theorem}

Theorem~\ref{mombo} yields approximations for objects which are finitely resolved by a class of distinguished objects. The next  approach we are going to present  uses  Milnor colimits rather than  iterated push-outs in order to  approximate arbitrary objects. But  to ensure a good behaviour with respect to Milnor colimits one still needs a finiteness condition, namely, one requires the  torsion pair to be compactly generated.

\begin{theorem}\label{AItorsionpair} Let $\T$ be a triangulated category with coproducts, and $\Scal$  a set of objects  in $\Tcal^c$. 

(1) \cite[Theorem 4.3]{AI} There is a torsion  pair $({}^{\perp_0}(\Scal^\perp),\Scal^{\perp_0})$ in $\Tcal$.

(2) \cite[Theorem A.1 and Lemma 3.1]{AJS},\cite[Theorem A.7]{KN},\cite[III, Proposition 2.8]{BR} There is a t-structure 
$(\mathsf{Susp}(\Scal),\Scal^{\perp_{\le 0}})$ in $\Tcal$. If  $\Scal\subset\Scal^{\perp_{>0}}$, then  $\mathsf{Susp}(\Scal)\subset\Scal^{\perp_{>0}}$, with equality if and only if $\Scal$ generates $\Tcal$.
\end{theorem}

The assumption ``compact'' in Theorem~\ref{AItorsionpair} can be removed when $\T$ admits an enhancement allowing to exploit a small-object argument in Quillen's sense. Using the theory of cotorsion pairs in exact categories, it is shown in  \cite[Corollary 3.5]{SS} that every set of objects $\Scal$
in the stable category of an efficient Frobenius category generates a torsion  pair $({}^{\perp_0}(\Scal^\perp),\Scal^{\perp_0})$. We will need  the following special case.

\begin{theorem}\label{AJS}
 \cite[Proposition 3.4]{AJS}
Let $\Acal$ be a Grothendieck category. Every object $X$ in $\mathsf D(\Acal)$ gives rise to a t-structure $(\mathsf{Susp}(X), X^{\perp_{\le 0}})$ in $\mathsf D(\Acal)$. 
\end{theorem}

Neeman has recently extended Theorem~\ref{AJS} in \cite{Nee18} to the class of well-generated triangulated categories, a generalisation of compactly generated triangulated categories which encompasses categories admitting a nice model, such as derived categories of Grothendieck categories. His proof doesn't require the existence of a model. It uses a
  different ingredient  to produce  adjoint functors: 
Brown representability.  

\begin{definition}
A triangulated category $\T$ with coproducts  satisfies  \emph{Brown representability} if every contravariant cohomological functor $H:\T\rightarrow \Ab$ from $\T$ to the category of abelian groups $\Ab$ which takes coproducts in $\T$ to products in $\Ab$ is representable. A triangulated category $\T$ with products satisfies \emph{Brown representability for the dual} if every covariant cohomological functor $H:\T\rightarrow \Ab$ which takes products in $\T$ to products in $\Ab$ is representable.
\end{definition}
 
If $\T$ is compactly generated,  then it satisfies both Brown representability and Brown representability for the dual - see \cite{Kr3} for details. 

\smallskip

Here is the key result connecting Brown representability with the existence of adjoint functors.

\begin{theorem}\cite[Theorem 8.4.4.]{Neeman}\label{Brown} Let $\Scal$ and $\Tcal$ be triangulated categories. Assume that $\Scal$ satisfies Brown representability, and let $F: \Scal\to \Tcal$ be a triangulated functor which respects coproducts. Then $F$ has a right adjoint functor.
\end{theorem}

The next result is an application of Theorem~\ref{Brown} which yields the existence of stable t-structures  (also called {semi-orthogonal decompositions}), that is, t-structures where the aisle and the coaisle are even triangulated subcategories of the original category.  Recall that in a triangulated category with coproducts,  a \emph{localising} subcategory is a  triangulated subcategory  which is closed under  coproducts, and which is then automatically  thick by Lemma~\ref{swindle}.  
Given a set of objects $\Xcal$ in $\Tcal$, we denote by $\mathsf{Loc}(\Xcal)$  the smallest  localising subcategory of $\T$ containing $\Xcal$.  
\begin{theorem}\cite[Proposition 3.8]{CGR} \label{CGR} Let $\Tcal$ be a compactly generated (or more generally,  a well-generated) triangulated category. Every object $X$ in $\Tcal$ gives rise to  a  t-structure $(\mathsf{Loc}(X), X^{\perp_{\mathbb Z}})$ in $\T$. 
In particular, an object $X$ is a generator of $\T$ if and only if $\mathsf{Loc}(X)=\T$.
\end{theorem} 

\subsection{Adjacent torsion pairs} We now turn to triples of classes formed by two torsion pairs.
\begin{definition}
Given two torsion pairs of the form $(\Ucal,\Vcal)$ and $(\Vcal,\Wcal)$, we say that $(\Ucal,\Vcal)$ is \emph{left adjacent} to $(\Vcal,\Wcal)$ and that $(\Vcal,\Wcal)$ is \emph{right adjacent} to $(\Ucal,\Vcal)$. The triple $(\Ucal,\Vcal,\Wcal)$ is then said to be a \emph{TTF (torsion-torsion-free) triple},
and it is said to be 
\begin{itemize}
\item
\emph{suspended} (respectively, \emph{cosuspended}) if so is the class $\Vcal$.
  \item\emph{generated by a set} of objects $\Scal$  if $\Vcal=\Scal^{\perp_0}$;  \item\emph{compactly generated} if it is generated by a set of compact objects.
 \end{itemize}
\end{definition}
In other words, a  triple $(\Ucal,\Vcal,\Wcal)$ is suspended if and only if $(\Ucal,\Vcal)$ is a co-t-structure  and $(\Vcal,\Wcal)$ is a t-structure, and the corresponding result  holds true for  {cosuspended} TTF triples.
A suspended TTF triple $(\Ucal,\Vcal,\Wcal)$  will be called \emph{nondegenerate} if so is the t-structure $(\Vcal,\Wcal)$. A cosuspended TTF triple $(\Ucal,\Vcal,\Wcal)$  will be called \emph{nondegenerate} if so is the t-structure $(\Ucal,\Vcal)$.

\smallskip

In presence of an adjacent co-t-structure, nondegeneracy can be rephrased as follows (cf.~\cite[Lemma 4.6]{AMV3} for the dual statement).

\begin{lemma}\label{generatingcoheart} 
A suspended TTF triple $(\Ucal,\Vcal,\Wcal)$ in $\T$ is  nondegenerate if and only if the coheart $\Ccal = \Ucal[1] \cap\Vcal$ generates $\T$. In this case, we have $\Vcal =\Ccal^{\perp_{>0}}$.
\end{lemma}

The next result gives criteria for the existence of adjacent torsion pairs.

\begin{theorem} \cite[Proposition 1.4]{Nee} \label{N} 
Let $\T$ be a triangulated category with coproducts, and let $\Vcal$ be a subcategory  of $\T$ which is  suspended (respectively, cosuspended). Then $\Vcal$ is precovering (respectively, preenveloping)
if and only if the inclusion of $\Vcal$ in $\T$ has a  right (respectively, left) adjoint. 
\end{theorem}

Further results hold true under Brown representabilty.

\begin{theorem}\cite[Theorem 3.1.2 and Corollary 4.3.9]{Bondarko2016}\label{coBondarko}
Let $\Tcal$ be a triangulated category
satisfying Brown representability. 
Every 
compactly generated co-t-structure has a right adjacent t-structure. Every compactly generated t-structure in $\T$  has a right adjacent co-t-structure.
\end{theorem}

\begin{theorem}\cite[Theorem 3.2.4 and Corollary 3.2.6]{Bondarko2016}\label{Bondarko}
Let $\Tcal$ be a triangulated category with  coproducts, and let $\mathfrak{t}:=(\Ucal,\Vcal)$ be a  t-structure in $\Tcal$ such that the smallest localising subcategory containing $\Ucal$ satisfies Brown representability for the dual. Then $\mathfrak{t}$ has a left adjacent co-t-structure if and only if $\mathfrak{t}$ is cosmashing and the heart $\Hcal_\mathfrak{t}$ has enough projectives. The dual statement holds true as well.
\end{theorem}


\section{Variants of silting} 

We are now ready for introducing silting objects. We will discuss and compare different variants of this notion  and explain the connection with torsion pairs and TTF triples. The crucial feature will be a self-orthogonality condition: a silting object does not have positive self-extensions, that is, it maps trivially to its positive shifts. This is a weaker form of the self-orthogonality required for a tilting object, which has to be \emph{exceptional}, i.e.~it has to map trivially to all its nonzero shifts. 

\subsection{Silting subcategories}
 
We start out with the notion 
introduced by Aihara and Iyama.

\begin{definition}\cite[Definition 2.1]{AI} \label{AI}
A  subcategory $\Scal$ of $\T$ is said to be a \emph{presilting subcategory} if $\Scal=\add\Scal\subset\Scal^{\perp_{>0}}$. It is \emph{silting} if, in addition, $\T=\mathsf{thick}(\Scal)$. 
\end{definition}
\begin{example}\label{compex}\cite[Example 2.2 and Proposition 2.20]{AI}
 Any ring $A$, viewed as a stalk complex, defines a silting subcategory in $\mathsf{K}^b(\mathrm{proj}(A))$. Moreover, 
 every silting subcategory  in $\mathsf{K}^b(\mathrm{proj}(A))$ is of the form  
 $\Scal=\add(T)$ for an object $T$.
 \end{example}

The existence of a silting subcategory is a strong requirement on $\T$. It imposes bounds on the distance between  objects with non-trivial morphisms, and it entails that all objects can be constructed from the silting subcategory by a finite number of shifts and extensions.

\begin{proposition}\label{require} \cite[Propositions 2.4 and 2.17, Example 2.5]{AI} Let  $\T$ be a triangulated category with  a silting  subcategory $\Scal$.
\begin{enumerate}
\item  
For any two objects $X,Y\in\T$ we have   $\Hom_\T(X,Y[i])=0$ for $i\gg 0$.

\item Every object in $\Tcal$  occurs as a summand of an object in an iterated extension  of the form $ \Scal[-\ell]\ast\Scal[1-\ell]\ast\ldots\ast\Scal[\ell-1]\ast\Scal[\ell]$ for some ${\ell\ge 0}$.

\item If $\T=\mathsf{D}^b(\modr A)$ for a  finite dimensional algebra $A$ over a field $k$, then  $A$ has finite global dimension.
\end{enumerate}
\end{proposition}

Silting subcategories are closely  related with  co-t-structures. The following  result was proved in \cite[Theorem 5.5]{MSSS}, see also   \cite{Bondarko2010}, \cite[Proposition 2.23]{AI},  \cite{KN2}, and \cite[Proposition 2.8]{IY}. 
\begin{theorem} \label{MSSS}
Every    silting subcategory $\Scal$ in $\T$ induces a bounded co-t-structure $\mathfrak t =(\Ucal, \Vcal)$ in $\T$ where 
\begin{enumerate}
\item $\Vcal=\Scal^{\perp_{>0}}$  consists of all objects that are finitely resolved by $\Scal$,   
and it 
is the smallest suspended subcategory of $\T$ containing $\Scal$;
\item  $\Ucal$  consists of all objects that are finitely coresolved by $\Scal[-1]$,
and it is the smallest cosuspended subcategory of $\T$ containing $\Scal[-1]$;
\item $\Scal=\Ucal[1]\cap\Vcal$ is the coheart of $\mathfrak t$.
\end{enumerate}
\end{theorem}
The statement is an application of  Theorem~\ref{mombo} to the pair $(\Xcal, \omega)$, where $\Xcal$ denotes the smallest cosuspended subcategory of $\T$ closed under direct summands and containing $\Scal$, and $\omega=\Scal$. 
It follows that $\Ucal=\Xcal[-1]$ and $\Vcal=\widehat{\Scal}$ form a co-t-structure   with coheart $\Scal$ in  $\widehat{\Xcal}=\mathsf{thick}(\Scal)=\Tcal$.

An alternative proof can be found in \cite[Proposition 2.8]{IY}. The main step there consists in showing that iterated extensions of the form $ \Scal\ast\Scal[1]\ast\ldots\ast\Scal[\ell]$ are closed under summands. The statement is then deduced from Proposition~\ref{require}(2).

\smallskip

To every silting subcategory $\Scal$ one can thus assign  
a bounded co-t-structure with coheart $\Scal$, and this assignment is obviously injective. On the other hand,  any bounded co-t-structure $(\Ucal,\Vcal)$ is determined by its coheart: indeed,  $\Vcal$ consists of the objects which are finitely resolved by the coheart.   Since the coheart is always a silting subcategory, we obtain the following result. 

\begin{corollary}\label{msss}\cite[Corollary 5.9]{MSSS}
There is a   bijective correspondence between 
silting subcategories  and bounded co-t-structures in $\T$.\end{corollary}

The partial order on co-t-structures given by inclusion of the (co)aisles induces a partial order on the collection $\mathsf{silt}\,\T$ of all silting subcategories of $\T$ defined by $$\Scal_1\ge\Scal_2\quad\text{ if }\quad \Scal_1^{\perp_{>0}}\supseteq\Scal_2^{\perp_{>0}}.$$
 In  \cite{AI}, Aihara and Iyama introduce  silting mutation and  show that it is closely related to the  partial order $\ge$. For instance, when $\T=\mathsf{K}^b(\mathrm{proj}(A))$ for a finite dimensional algebra $A$, silting mutation is depicted by the Hasse quiver of the poset $(\mathsf{silt}\,\T,\ge)$.
 This quiver is known to be connected - that is, iterated silting mutation is transitive -  whenever $A$ is a local, or a hereditary, or a canonical algebra \cite{AI}, or belongs to the class of silting-discrete algebras which will be discussed in Section~\ref{discrete}. Some examples of such quivers are computed in \cite[Section 2.6]{AI}. More details are given in \cite{BY}.
  
 The poset of silting subcategories can also be used to revisit a reduction process, called Calabi-Yau reduction, which is  employed in cluster tilting theory \cite{IYo}. In \cite{IY}, Iyama and Yang establish an isomorphism between the poset formed by the silting subcategories of a triangulated category $\T$ which contain a given presilting subcategory $\Pcal$,  and the poset of silting subcategories of the quotient $\T/\mathsf{thick}(\Pcal)$. Moreover, they discuss how this process, called \emph{silting reduction}, interacts with Calabi-Yau reduction.

\subsection{Large silting objects}

Aihara and Iyama observed in \cite[Section 4]{AI} that  silting is also closely related with  t-structures. To establish a connection, however, they had to move to triangulated categories with coproducts and to restrict to silting subcategories consisting of compact objects. 

This interplay with t-structures  led  Psaroudakis and Vit\'oria \cite{PV}, and independently, Nicol\'as, Saor\'in, and Zvonareva \cite{NSZ},  to develop the following concept where the requirement of  compactness is dropped (cf.~\cite[Definition 4.1]{PV} and \cite[Remark 3]{NSZ}).

\begin{definition}\label{PV} 
Let $\T$ be a triangulated category with coproducts.
An object $T$ in $\T$ is called \emph{silting} if $(T^{\perp_{>0}},T^{\perp_{\leq 0}})$ is a t-structure in $\T$.
We call such a t-structure \emph{silting}, denote its  heart  by $\Hcal_T$, and write $H^0_T:\T\longrightarrow \Hcal_T$ for the associated cohomological functor.
\end{definition}

It follows immediately from the definition  that any silting object $T$ generates $\T$ and is contained in $T^{\perp_{>0}}$. Moreover, $T^{\perp_{>0}}$ is clearly closed under products. Hence the silting  t-structure given by $T$ is cosmashing, and it is  nondegenerate, because  the intersection of the shifted aisles $\bigcap_{k\in\mathbb Z} T^{\perp_{>0}}[k]$ coincides with  $T^{\perp_{\mathbb Z}}=0$, and similarly for the shifted coaisles.

\begin{examples}\label{relation} 
(1) Let $\T$ be a compactly generated triangulated category.
A family of compact objects $(T_k)_{k\in K}$ forms a silting subcategory  $\Scal=\add\{T_k\,\mid\, k\in K\}$ in $\T^c$  
if and only if the coproduct $T=\textstyle \coprod_{k\in K} T_k$ is a silting object in $\T$.
Indeed, a silting subcategory $\Scal$ of $\T^c$ generates $\T$, and by Theorem \ref{AItorsionpair}(2), it induces a t-structure $(\Scal^{\perp_{>0}}, \Scal^{\perp_{\le 0}})$ in $\T$. 
This shows the only-if-part. The reverse implication follows from the fact that
the thick closure of a compact and generating subcategory of $\T$  always equals
$\T^c$, 
cf.~\cite[Proposition 4.2]{AI}. 

\smallskip

(2) \cite[Proposition 4.13]{PV} Let $\T=\mathrm{D}(\Acal)$ be the derived category of  a Grothendieck category $\Acal$. A complex  $T$ is a silting object in $\T$ if and only if $T$ lies in $T^{\perp_{>0}}$ and generates $\T$, and  $T^{\perp_{>0}}$ is closed under coproducts.
\end{examples}

Definition~\ref{PV} relaxes the strong generation condition from Definition~\ref{AI} which requires $\T$ to be the thick closure of $\Scal$.   
This removes constraints on the existence of non-trivial morphisms and on the shape of   arbitrary objects from $\T$, as illustrated by
   the next result. 
   Roughly speaking,
  iterated extensions of shifts are now replaced by  countable extensions which are  constructed as Milnor colimits.



\begin{proposition}\label{properties silting} \cite[Section 4]{PV}\cite[Theorem 2]{NSZ} Let $\T$ be a triangulated category with coproducts, and let   $T$ be a silting object in $\T$ with corresponding silting t-structure $(\Vcal, \Wcal)$. Then 
\begin{enumerate} 
\item  
$\Vcal
= \mathsf{Susp}(T)$ is the smallest aisle of $\T$ containing $T$;
\item for every object $X\in\T$ the truncation $v(X)$ is  a Milnor colimit of a sequence of the form $$\xymatrix{V_0\ar[r]^{f_1}&V_1\ar[r]^{f_2}&V_2\ar[r]&\ldots}$$ where $V_0$ is a coproduct of copies of $T$, and for each $n$ the cone of  $f_n$ is a coproduct of copies of $T[n]$;
\item the orthogonal category $\Ucal={}^{\perp_0}\Vcal$ satisfies $\Ucal[1]\cap\Vcal=\Add(T)$.
\end{enumerate}
\end{proposition}
\begin{proof} 
First of all, notice that the Milnor colimits described in (2) belong to $\mathsf{Susp}(T)$, hence (1) follows from (2).
 For statement (2), we sketch the idea in the proof of \cite[Theorem 2]{NSZ}. Given  $X\in\Tcal$, one constructs a direct system of triangles 
  \begin{equation}\label{triangles}
\Delta_n:\quad V_n\stackrel{a_n}{\longrightarrow} X \stackrel{b_n}{\longrightarrow} Y_n\longrightarrow V_n[1]
\end{equation}
with the property that $V_n$ belongs to $\Vcal$ and
\begin{equation}\label{orthogonal}
\Hom_\T(T[k],Y_n)=0\:\text{ for all } 0\le k\le n.\end{equation}
For $n=0$ the map $a_0:V_0{\longrightarrow} X$ is just the universal morphism defined by all morphisms $T\to X$. If $n>0$ and the triangles $\Delta_0,\ldots,\Delta_{n-1}$  are constructed, then one considers the universal morphism $T_n\longrightarrow Y_{n-1}$ defined by all morphisms $T[n]\to Y_{n-1}$ and observes that its cone $Y_n$ satisfies condition (\ref{orthogonal}). From   the triangle $\Delta_{n-1}$, using the octahedral axiom, one obtains a triangle $\Delta_n$ together with a connecting morphism $f_n:V_{n-1}\longrightarrow V_n$ with cone $T_n$.

Now let $V$ be the Milnor colimit of the sequence $V_0\stackrel{f_1}{\longrightarrow}V_1\stackrel{f_1}{\longrightarrow}V_2\stackrel{f_3}{\longrightarrow}\ldots$. There is a triangle 
\begin{equation}
\coprod_{n\ge 0}V_n\longrightarrow \coprod_{n\ge 0}V_n\stackrel{p}{\longrightarrow}V\longrightarrow\coprod_{n\ge 0}V_n[1]
\end{equation}
 showing that $V$ belongs to $\Vcal$, and by construction,  there is a morphism $V\stackrel{f}{\longrightarrow} X$ such that $f\circ p$ is the morphism $a$ induced by the family $(a_n)$. We obtain a triangle
\begin{equation}\label{triangle3}
V\stackrel{f}{\longrightarrow} X\longrightarrow W\longrightarrow V[1]\end{equation}
and it remains to show that $W$ belongs to $\Wcal$. Fix $k\ge 0$. By condition (\ref{orthogonal}), the morphism  $\Hom_\T(T[k],a_n)$ is surjective for all $n\ge k$. Since $a_n=f\circ p\mid_{V_n}$, it follows that the morphism  $\Hom_\T(T[k],f)$ is surjective. For $k=0$, this immediately implies that  $\Hom_\T(T,W)=0$. But in fact, one can show that the construction also entails injectivity of  $\Hom_\T(T[k],f)$. From the triangle (\ref{triangle3}) one then infers that $\Hom_\T(T[k],W)=0$ for all $k\ge 0$, hence 
$W$ belongs to $T^{\perp_{\le 0}}=\Wcal$. 
\smallskip

Finally, let us check (3). The inclusion $\supset$ is clear. For the other inclusion, given an object $X$ in $\Ucal[1]\cap\Vcal$, we consider the triangle
  $K\to T'\stackrel{u}{\longrightarrow} X\stackrel{v}{\longrightarrow} K[1]$
  where $u$ is 
  an $\Add(T)$-precover of $X$. Applying the functor $\Hom_{\T}(T,-)$, we see that $\Hom_{\T}(T,K[i])=0$ for all $i>0$, hence $K\in\Vcal$ and therefore $\Hom_\T(X,K[1])=0$. In particular, $v=0$ and the triangle splits, so 
$X$ lies in $\Add(T)$.
\end{proof} 
 
Let us now investigate the connection with co-t-structures.

\begin{proposition}\label{silting heart}
Let $\T$ be a   triangulated category with coproducts, and let $T$ be a silting object in $\T$ with corresponding silting t-structure $\mathfrak{t}=(\Vcal, \Wcal)$.
Then \begin{enumerate}
\item  $H^0_T(T)$ is a projective generator of the heart $\Hcal_T$. In particular, if $T$ is compact, $\Hcal_T$ is equivalent to the category of modules over 
$\End_\T(T)$.
\item If  $\T$ is compactly generated, then $\mathfrak{t}$  has a left adjacent co-t-structure $(\Ucal,\Vcal)$.
\end{enumerate}
\end{proposition}
\begin{proof}  (1) For the reader's convenience we repeat the arguments from \cite[Proposition 4.3 and Corollary 4.7]{PV}. Recall that  $\Hcal_T=\Vcal\cap\Wcal[1]$ and $H^0_T(T)=W[1]$ in the canonical triangle
\begin{equation}\label{triangle}
V[1]\longrightarrow T \stackrel{f}{\longrightarrow} W[1]\longrightarrow V[2]
\end{equation} with $V\in\Vcal$ and $ W\in \Wcal$. 
Given an object $X$ in the heart $\Hcal_T$, one  verifies by applying $\Hom_{\T}(-,X)$ on (\ref{triangle}) that $\Hom_\T(H^0_T(T),X)\cong \Hom_\T(T,X)$  and
$\Ext^1_{\Hcal_T}(H^0_T(T), X)\cong\Hom_\T(H^0_T(T), X[1])=0$. Hence  $H^0_T(T)$ is a projective object in  $\Hcal_T$. 
Moreover, since $\Hcal_T=T^{\perp_{\not= 0}}$ and $T$ is a generator of $\T$, it follows 
that $\Hom_\T(H^0_T(T),X)\not=0$ whenever $X\not=0$, showing that  $H^0_T(T)$  is also a generator in the heart $\Hcal_T$.
Finally, compactness of $T$ ensures that the projective generator  $H^0_T(T)$  is small, so it follows from  Morita theory that $\Hcal_T$ is equivalent to the category of modules  over $B=\End_{\T}(H^0_T(T))\cong\End_{\T}(T)$. 
 
(2) Using Theorem~\ref{CGR} and the fact that $T$ is a generator, we see that  the smallest localising subcategory of $\T$ containing the aisle $\Vcal$ of $\mathfrak t$ is $\T$ itself, and thus it satisfies  {Brown representability for the dual}. Moreover,  $(\Vcal,\Wcal)$ is cosmashing t-structure whose heart $\Hcal_\mathfrak{t}$ has enough projectives by (1).  Now it follows from Theorem~\ref{Bondarko} that  $\mathfrak t$ has a left adjacent co-t-structure.
\end{proof}


 By condition (3) in Proposition~\ref{properties silting}, the aisle of a silting t-structure determines the additive closure $\Add(T)$ of the silting object. This suggests the following
 \begin{definition} 
  Two  silting objects $T,T'$  in $\T$ are \emph{equivalent} if 
   $T^{\perp_{>0}}= T'\,^{\perp_{>0}}$.\end{definition}
We are now ready to prove that silting objects parametrise  t-structures and TTF triples in $\T$, cf.~\cite[Section 4]{NSZ}.

\begin{theorem}\label{NSZ-bijection} If $\T$ is a compactly generated triangulated category, there is a   bijective correspondence between 
\begin{enumerate}
\item[(i)] equivalence classes of silting objects;
\item[(ii)] nondegenerate cosmashing t-structures   whose heart admits a projective generator;
\item[(iii)] nondegenerate suspended TTF triples
which are {generated by a set}.
\end{enumerate}
\end{theorem}
\begin{proof} 
From the discussion above we already know that every silting object $T$ gives rise to a \mbox{t-structure} $(\Vcal,\Wcal)$ as in (ii) and a suspended TTF triple $(\Ucal,\Vcal,\Wcal)$ as in (iii), where $\Vcal=T^{\perp_{>0}}$  is determined by the equivalence class of $T$. Hence we have injections (i)$\to$(ii), and  (i)$\to$(iii).

Conversely, if $\mathfrak{t}=(\Vcal,\Wcal)$ is a nondegenerate cosmashing t-structure whose heart admits a projective generator $P$, then the cohomological functor $\Hom_T(P,H^0_\mathfrak{t}(-))$ takes products to products. So, as  $\T$ satisfies  {Brown representability for the dual},  there exists an object $T\in\T$ such that $\Hom_T(P,H^0_\mathfrak{t}(-))=\Hom_\T(T,-)$, and it turns out that $T$ is a silting object with silting t-structure $\mathfrak{t}$. This proves that the map  (i)$\to$(ii) is bijective.

Consider now a nondegenerate suspended TTF triple $(\Ucal,\Vcal,\Wcal)$ which is generated by a set $\Scal$ (which then has to lie in $\Ucal$). For every object $S\in\Scal$ we consider the canonical triangle $$U_S\longrightarrow S[1]\stackrel{f_S}{\longrightarrow} V_S\longrightarrow U_S[1]$$ with $U_S\in\Ucal$ and $V_S\in\Vcal$.
Since $S[1]$ and  $U_S[1]$ lie in $\Ucal[1]$, the object $V_S$ lies in the coheart $\Ccal = \Ucal[1] \cap\Vcal$. Setting $T=\coprod_{S\in\Scal} V_S$, it follows that $\Add(T)\subset \Ccal$. To show the reverse inclusion, we pick an object $C\in\Ccal$ together  with a triangle
  $$K\to T'\stackrel{u}{\longrightarrow} C\stackrel{v}{\longrightarrow} K[1]$$ 
  where $u$ is  an $\Add(T)$-precover of $C$, and we apply the functor $\Hom_\Tcal(S[1],-)$ on it. Using that $f_S$ is a left $\Vcal$-approximation of $S[1]$, we see that $\Hom_\Tcal(S[1],u)$ is surjective.
  As in the proof of Proposition~\ref{properties silting}, 
  we conclude that 
  $K$ lies in $\Vcal$, and the triangle splits,
  as desired. 
  
  Now we infer from Lemma~\ref{generatingcoheart} that  $T$ is a generator, and in fact it is  a silting object in the preimage of the  TTF triple $(\Ucal,\Vcal,\Wcal)$. So, we conclude that also the 
map  (i)$\to$(iii) is bijective.
\end{proof}

\section{Silting complexes}\label{Siltcompl}

\subsection{Silting objects in derived module categories}
Throughout this section, let $A$ be a ring with identity. According to Example~\ref{relation}, the regular module  $A$, viewed as a stalk complex, is a silting object in the unbounded derived category $\T=\mathsf{D}(A)$. The corresponding t-structure is the  standard t-structure $(\mathsf D^{\leq 0}, \mathsf D^{>0})$ from Example~\ref{standard}.

In order to find more 
silting objects in $\mathsf{D}(A)$, we turn to the complexes which were termed ``big semitilting'' by Wei in \cite{Wei}.

\begin{definition}\label{silting complex} A bounded complex of projective $A$-modules $\sigma$ in $\mathsf{K}^b(\mathrm{Proj}(A))$ is a \emph{silting complex} if it satisfies the following conditions:
\begin{enumerate}
\item[(S1)] $\Hom_{D(A)}(\sigma,\sigma^{(I)}[i])=0$ for all sets $I$ and $i>0$;
\item[(S2)] the smallest triangulated subcategory of $\mathsf D(A)$ containing $\Add(\sigma)$ is $\mathsf{K}^b(\mathrm{Proj}(A))$.
\end{enumerate}
Moreover, $\sigma$ is a \emph{tilting complex} if,  in addition, it  is compact and exceptional, i.e.~it belongs to  $\mathsf{K}^b(\mathrm{proj}(A))$ and satisfies $\Hom_{D(A)}(\sigma,\sigma[i])=0$ for all $i\not=0$.
\end{definition}

\begin{remark}\label{commentstodef}
(1) From Lemma~\ref{swindle} one easily deduces that
 the smallest triangulated subcategory of $\mathsf D(A)$ containing $\Add(\sigma)$ is closed under direct summands. Hence condition (S2) actually requires $\mathsf{K}^b(\mathrm{Proj}(A))=\mathsf{thick}(\Add\sigma)$. In other words,
 $\sigma$ is a silting complex if and only if $\Add\sigma$ is a silting subcategory of $\mathsf{K}^b(\mathrm{Proj}(A))$.

\smallskip

(2)
If $\sigma$ belongs to  $\mathsf{K}^b(\mathrm{proj}(A))= \mathsf D (A)^c$, then condition (S1) is equivalent to the condition $\Hom(\sigma,\sigma[i])=0$ for all $i>0$, and
(S2) amounts to $\sigma$ being a generator. The silting subcategories of $\mathsf{K}^b(\mathrm{proj}(A))$ are thus precisely the subcategories of the form $\Scal=\add\sigma$ for a compact silting complex $\sigma$, cf.~Example~\ref{compex}. If $A$ is a finite dimensional algebra, then  one can choose $\sigma$ to be \emph{basic}, i.e.~a direct sum of $n$ pairwise non-isomorphic indecomposable compact complexes. Here $n$ is the number of isomorphism classes of indecomposable projective $A$-modules, see \cite[Corollary 2.28]{AI}.
\end{remark}

By condition (S1), the class $\sigma^{\perp_{>0}}$ given by a silting complex $\sigma$ contains $\mathsf{Susp}(\sigma)$. 
Conversely, given an object $X$ in $\sigma^{\perp_{>0}}$, we can consider the canonical triangle 
$V\longrightarrow X\longrightarrow W \longrightarrow V[1]$ with respect to the t-structure  $(\mathsf{Susp}(\sigma),\sigma^{\perp_{\le 0}})$ from Theorem~\ref{AJS}. Then $W$ lies in $\sigma^{\perp_{\le 0}}$, but also in $\sigma^{\perp_{>0}}$ as so do $X$ and $V[1]$. Since $\sigma$ is a generator, we infer that $W=0$ and $X$ belongs to   $\mathsf{Susp}(\sigma)$.
We have thus shown that $(\sigma^{\perp_{> 0}},\sigma^{\perp_{\le 0}})$ is a t-structure, and $\sigma$ is a silting object in $\mathsf{D}(A)$.

\begin{proposition}\cite[Proposition 4.2]{AMV1}\label{obj=cpx} A bounded complex of projective $A$-modules  is a silting object in $\mathsf{D}(A)$ if and only if it is a {silting complex}.
\end{proposition}

To prove the only-if part in Proposition~\ref{obj=cpx}, we  consider  a bounded complex of projectives  which, for simplicity, is assumed to
be concentrated between degrees $-n$ and $0$. If  $$\sigma:\:\cdots \longrightarrow 0\longrightarrow P_{-n}\longrightarrow \cdots\longrightarrow P_{-1}\longrightarrow P_0\longrightarrow 0 \longrightarrow\cdots $$ is a  silting object in $\mathsf{D}(A)$, then the corresponding suspended TTF triple $(\Ucal,\Vcal,\Wcal)$ satisfies  $$\mathsf D^{\leq -n}\subset \Vcal=\sigma^{\perp_{>0}}=\mathsf{Susp}(\sigma)\subset\mathsf D^{\leq 0}.$$
The next result, which goes back to \cite[Proposition 3.12]{Wei}, will prove that $\sigma$ then  satisfies condition (S2). Indeed, we are going to see that the ring $A$, viewed as a stalk complex concentrated in degree 0,  lies in the smallest triangulated subcategory of $\mathsf D(A)$ containing the coheart $\Ucal[1]\cap\Vcal$,  and the latter equals  $\Add(\sigma)$ by Proposition~\ref{properties silting}. 

\begin{lemma}\label{wei} Let  $(\Ucal,\Vcal,\Wcal)$ be a suspended   TTF triple such that  $\mathsf D^{\leq -n}\subset \Vcal\subset\mathsf D^{\leq 0}$ for some  $n>0$. 
Then $A$ is  finitely coresolved by objects of the coheart, i.e.~there is a finite sequence of triangles 
$$Z_i\longrightarrow C_i\longrightarrow Z_{i+1}\longrightarrow Z_i[1],\quad 0\le i<n,$$
where $A=Z_0$, and the objects  $C_0,\ldots, C_{n-1}, Z_n$  lie in the coheart $\Ccal=\Ucal[1]\cap\Vcal$.
\end{lemma}
\begin{proof}
First of all, the inclusion $\mathsf D^{\leq -n}\subset \Vcal$ implies  that $A[n]$ lies in $\Vcal$, and the inclusion $\Vcal\subset\mathsf D^{\leq 0}$ implies that $A$ lies in $\Ucal[1]$.  
The co-t-structure  $(\Ucal,\Vcal)$ then yields a sequence of canonical triangles 
$$\xymatrix{Z_i\ar[r]^{}&C_i\ar[r]^{}&Z_{i+1}\ar[r]&Z_i[1]},\quad i\ge 0,$$
where $A=Z_0$, the objects $Z_i$ belong to  $\Ucal[1]$, and the objects $C_i$  lie in the coheart. 
One checks that $$\Hom_\T(Z_{n+1},Z_n[1])\cong\Hom_\T(Z_{n+1},Z_{n-i}[i+1])$$ for all $1\le i\le n$, and in particular $\Hom_\T(Z_{n+1},Z_n[1])\cong\Hom_\T(Z_{n+1},A[n+1])=0$. This shows that the  $n$-th triangle splits, hence $Z_n$ lies in the coheart as well.
\end{proof}


We will see in Example~\ref{unbounded} that silting objects in $\mathsf{D}(A)$  need not be bounded. As shown in \cite[Proposition 4.17]{PV}, a silting object  is a bounded complex of projectives, that is, a silting complex in the sense of Definition~\ref{silting complex}, if and only if the 
corresponding TTF triple has the following property.

\begin{definition}\label{interm}
A  t-structure $(\Vcal,\Wcal)$, or a suspended TTF triple $(\Ucal,\Vcal,\Wcal)$ 
in $\mathsf D(A)$ is said to be \emph{intermediate} if there are integers $n\leq m$ such that $\mathsf D^{\leq n}\subseteq \Vcal\subseteq \mathsf D^{\leq m}.$
\end{definition}

Observe that intermediate suspended TTF-triples are nondegenerate as so is the standard t-structure.
We recover   the following result,  cf.~also \cite[Theorem 5.3]{Wei}. 

\begin{theorem}\label{bijections general}\cite[Theorem 4.6]{AMV1}
There is a bijection between
\begin{enumerate}
\item[(i)] equivalence classes of (bounded) silting complexes in $D(A)$;
\item[(ii)] intermediate suspended TTF triples in  $\mathsf D(A)$.\end{enumerate}
\end{theorem}
In order to deduce this statement from Theorem~\ref{NSZ-bijection}, we have to show that the bijection between silting objects and the suspended TTF triples 
  considered there  
restricts to a bijection between (i) and (ii). This amounts to showing that every intermediate suspended  TTF-triple $(\Ucal,\Vcal,\Wcal)$  is  generated by a set. Recall from Lemma~\ref{wei} that $A$ is coresolved by finitely many objects  $C_0,\ldots, C_{n-1}, C_n$  from the coheart $\Ccal=\Ucal[1]\cap\Vcal$. The object $T=\bigoplus_{i=0}^n C_i$ then satisfies $\Add(T)=\Ccal$. 
This is shown as in the proof of Proposition~\ref{properties silting}. Now we infer from Lemma~\ref{generatingcoheart} 
that   $T$ is a silting complex corresponding to $(\Ucal,\Vcal,\Wcal)$  under our bijection.

\smallskip

In fact, intermediate suspended TTF triples in  $\mathsf D(A)$ are not only generated by  silting complexes, but also  by sets of compact objects. This ``finite type'' property  is inherited from the corresponding property of tilting modules. It shows that, even though our silting complexes need not be compact, their equivalence classes are determined by compact objects.

\begin{theorem}\label{finite type}\cite[Theorem 3.6]{MV} Every intermediate suspended TTF triple in  $\mathsf D(A)$ is compactly generated.
\end{theorem}
The proof of Theorem~\ref{finite type} uses the fact that one can view  an $n$-term  complex  as a representation in $\Modr A$ of the linearly oriented  Dynkin quiver $\mathbb A_n:1\to 2\to\ldots\to n$ bound by the relations given by all paths of length two, 
hence as a module over a quotient of the path algebra of $\mathbb A_n$. This allows to
 interpret silting complexes as  tilting modules in suitable module categories. The theorem now follows from an important result in tilting theory 
 stating
that  every tilting class is the Ext-orthogonal class  of a set of finitely presented modules of bounded projective dimension.

\medskip

We have seen that silting complexes parametrise  torsion pairs in $\mathsf D(A)$. There are similar results in further triangulated categories related to $A$. For example, combining Remark~\ref{commentstodef}(1) with Corollary~\ref{msss}, one   obtains a bijection between equivalence classes of silting complexes in $\mathsf D(A)$ and bounded co-t-structures in $\mathsf{K}^b(\mathrm{Proj}(A))$ whose coheart has an additive generator. As shown in \cite[Remark 4.7]{AMV1}, the bounded co-t-structure associated to a silting complex $\sigma$ under this bijection is precisely the restriction to $\mathsf{K}^b(\mathrm{Proj}(A))$ of 
 the co-t-structure $(\Ucal,\Vcal)$ in the intermediate suspended  TTF triple $(\Ucal,\Vcal,\Wcal)$ of Theorem~\ref{bijections general}.

Assume now that $\sigma$ is a compact silting complex over  a finite dimensional algebra $A$. Then  the co-t-structure $(\Ucal,\Vcal)$ also
restricts to a  bounded co-$t$-structure in  $\mathsf{K}^b(\mathrm{proj}(A))$,  and  the t-structure $(\Vcal,\Wcal)$
restricts to  a bounded $t$-structure in $\mathsf D^b(\fdmod A)$. 
Moreover,  by Proposition~\ref{silting heart},  the heart $\Hcal_\sigma$ 
is equivalent to the category of modules over the finite dimensional algebra $\End_{\mathsf D(A)}(\sigma)$. It follows that the heart of the bounded $t$-structure in $\mathsf D^b(\modr A)$ is a \emph{length category}, i.~e.~a skeletally small abelian category in which every object admits a finite filtration by simple objects. 
We thus recover the following  result due to K\"onig and Yang \cite{KY}, and to Keller and Nicol\'as \cite{KN,KN2}.

\begin{theorem}\label{thm:ky-bijection} 
Over  a finite dimensional algebra $A$, there are bijections between 
\begin{enumerate}
\item
equivalence classes of compact silting complexes in $\mathsf D(A)$;
\item
bounded co-$t$-structures in  $\mathsf{K}^b(\mathrm{proj}(A))$;
\item
bounded $t$-structures in $\mathsf D^b(\modr A)$ whose heart is a length category. 
\end{enumerate}
\end{theorem}

In fact, the bounded t-structure in (3) is determined by its  heart, which in turn, being a length category, is determined by a set of pairwise non-isomorphic  simple objects. These objects form a \emph{simple minded collection} in $\mathsf D^b(\modr A)$, and in  \cite{KY,KN}   a further bijection is established with 
\begin{enumerate}
\item[(4)] {\it isomorphism classes of {simple minded collections} in} $\mathsf D^b(\modr A)$.
\end{enumerate}
For further details we refer to \cite{SuYang} and to the survey \cite{BY}.

\subsection{Two-term silting complexes}

Let $P_{-1}\stackrel{\sigma}{\to} P_0$ be   a two-term complex in $\mathsf K^b(\mathrm{Proj}(A))$.
An object $X\in \Dcal^{\le 0}$ lies in $\sigma^{\perp_{>0}}$ if and only if all maps of complexes $\sigma\to X[1]$ are null-homotopic. This is expressed by the diagram below:    every map $f$ can be written as $f=s_0\,\sigma+d_{-1}\,s_{-1}$.
\[
\xymatrix{
	& 0\ar[r]&P_{-1} \ar[r]^{\sigma}\ar[d]^f\ar@{-->}[ld]_{s_{-1}} &
	P_0 \ar[r]\ar@{-->}[ld]_{s_{0}}  & 0\ar[r]& \dots \\
	\ar[r] & X_{-1} \ar[r]^{d_{-1}} & X_0 \ar[r]\ar[d]^{\pi} & 0 \ar[r] & 0 \ar[r] &  \dots \\
	&& H^0(X) 
}\]
It is easy to see that this means precisely that the map $\pi f$ factors through $\sigma$, or equivalently, every map $h:P_{-1}\to H^0(X)$  factors through $\sigma$. In other words,  $X$ lies in $\sigma^{\perp_{>0}}$ if and only if
 $H^0(X)$ belongs to the class $\Dcal_{\sigma}:=\{M\in \Modr A\mid \Hom(\sigma,M) \ \mbox{ is surjective}\}.$  

Consider now $T:=H^0(\sigma)$ with the exact sequence $\xymatrix{ P_{-1}\ar[r]^{\sigma} & P_0\ar[r] & T \ar[r] & 0}$, and denote by
\begin{itemize}
\item  $\mathsf{Gen}(T)$ the class of all modules that are epimorphic images of modules in $\Add(T)$;
\item $\Fcal$ the class of all $A$-modules $X$ such that $\Hom_A(T,X)=0$.
\end{itemize}
Using the observations above, one can show 
the following result, which is a non-compact version of  \cite[Theorem 2.10]{HKM}.

\begin{theorem}\cite[Theorem 4.9]{AMV1}\label{HKM}
The following statements are equivalent for a two-term complex $\sigma$ in $\mathsf K^b(\mathrm{Proj}(A))$ with  $T=H^0(\sigma)$.
\begin{enumerate}
\item
$\sigma$ is a silting complex; 
\item
$\mathcal D_{\sigma}=\mathsf{Gen}(T)$;
\item
$(\mathcal D_{\sigma},\Fcal)$ is a torsion pair in $\Modr A$. 
\end{enumerate}
Under these conditions, the HRS-tilt of the torsion pair $(\mathcal D_{\sigma}, \Fcal)$ is precisely
$(\sigma^{\perp_{>0}},\sigma^{\perp_{\le 0}})$. 
\end{theorem}

\begin{definition}\cite{AMV1}
A module $T$ is said to be a \emph{silting module} if it admits a projective presentation $\xymatrix{P_{-1} \ar[r]^{\sigma} & P_0}$ such that 
$\mathcal D_{\sigma}=\mathsf{Gen}(T)$. 
\end{definition}

The assignment $\sigma\mapsto H^0(\sigma)$ defines, up to equivalence, a bijection between two-term silting complexes in $\mathsf D(A)$ and  silting modules. The latter parametrise certain torsion pairs in $\Modr A$ and are intimately related with ring theoretic localisations of $A$. We refer to \cite{AMV1,AMV2,abundance} for details. 

If $A$ is a finite-dimensional algebra over a field, then the finite dimensional silting $A$-modules are precisely the support $\tau$-tilting modules introduced in
\cite{AIR}, and they parametrise certain torsion pairs in $\modr A$. Recall that a subcategory $\Xcal$ of $\modr A$ is {\it functorially finite} if every module in $\modr A$ has an $\Xcal$-preenvelope and an $\Xcal$-precover.

\begin{theorem}\cite[Theorem 3.2]{AIR}\label{AIRbij}
Let  $A$ be a finite-dimensional algebra over a field. There are bijections between
\begin{enumerate}
\item[(i)] equivalence classes of compact  two-term silting complexes in $\mathsf D(A)$;
\item[(ii)]  isomorphism classes of basic support $\tau$-tilting  $A$-modules;
\item[(iii)] functorially finite torsion classes  in $\modr A$.
\end{enumerate}
\end{theorem}


The partial order on functorially finite torsion classes in $\modr A$ given by inclusion induces a partial order on the isomorphism classes of basic support $\tau$-tilting modules. The poset  (s$\tau$-tilt$A$,\,$\ge$) obtained in this way encodes mutation by \cite[Corollary 2.34]{AIR}, and  it corresponds to a full subposet of $(\mathsf{silt}\,\mathsf{K}^b(\mathrm{proj}(A)),\ge)$ as a consequence of the interplay between  torsion pairs and their HRS-tilts, cf.~Theorem~\ref{HKM}.

 In \cite{Asai},  the  poset (s$\tau$-tilt$A$,\,$\ge$) is described in terms of certain collections of pairwise Hom-orthogonal \emph{bricks}, i.e.~modules whose endomorphism ring is a division ring. Such collections, called \emph{left finite semibricks}, can be regarded as a two-term version of the simple-minded collections from Theorem~\ref{thm:ky-bijection}. 
 
 Combinatorial descriptions of (s$\tau$-tilt$A$,\,$\ge$) are available for certain classes of algebras, e.g.~for Nakayama algebras, or for preprojective algebras of Dynkin type, see  \cite{Adachi,Adachi 2,  Asai2, BuM1, Kase, Ma, Mizuno, IZh}.
   
\begin{theorem}\cite{DIJ,AMV4}\label{DIJ} Let  $A$ be a finite-dimensional algebra over a field. The following statements are equivalent.
\begin{enumerate}
\item  The poset \mbox{\rm (s$\tau$-tilt$A$,\,$\ge$)}  is finite.
\item Every torsion class in $\modr A$ is functorially finite.
\item Every silting $A$-module is finite dimensional up to equivalence.
\end{enumerate}
\end{theorem}
\begin{definition}  Under the conditions of Theorem~\ref{DIJ}, the algebra $A$ is said to be \emph{$\tau$-tilting finite}.
\end{definition}

Examples of such algebras include all algebras of finite representation type, and all  local  finite dimensional algebras \cite[Lemma 3.5]{Ma}. More examples  will be discussed in Section~\ref{discrete} where we introduce the class of silting-discrete algebras.  

The poset of basic support $\tau$-tilting modules over a $\tau$-tilting finite algebra has a  connected Hasse quiver by \cite[Corollary 3.10]{AIR}, and it is a complete lattice by condition (2) in Theorem~\ref{DIJ}. 
In general, however, the poset (s$\tau$-tilt$A$,\,$\ge$)   is not  a lattice.  For example, for a path algebra $A=kQ$, it forms a lattice if and only if the quiver $Q$ is either a Dynkin quiver of type $\mathbb{A, D, E}$, or it has exactly two vertices, see  \cite{IRTT}.
Results  on the lattice of all torsion classes in $\modr A$ can be found e.g.~in \cite{DIRRT,BCZ}.

A  reduction process for support $\tau$-tilting modules which is  compatible with silting reduction is introduced in \cite{Jasso}. Further reduction theorems are given in \cite{Adachi, Eisele}. 
The notion of support $\tau$-tilting module has been extended to  additive categories   in \cite{IJY}. 

\subsection{Hearts and endomorphism rings} 
Observe that  
a module $T$ is silting with respect to a monomorphic projective presentation $\sigma:P_{-1}{\hookrightarrow} P_0$ if and only if $\mathsf{Gen}(T)=\Ker\Ext_A^1(T,-)$.
The modules with this property  are called {\it (1-)tilting} \and
 are characterised by the following conditions:
 \begin{enumerate}
\item[(T1)] $T$ has projective dimension at most one;
\item[(T2)] $\Ext^1_A(T,T^{(I)})=0$ for all sets $I$;
\item[(T3)] There is a coresolution $0\longrightarrow A\longrightarrow T_0\longrightarrow T_1\longrightarrow 0$ with $T_0,T_1\in \Add(T)$. 
\end{enumerate}
Finitely presented tilting modules are called \emph{classical}.

 \smallskip
 
One of the fundamental results of tilting theory, the Tilting Theorem published by Brenner and Butler in 1980, states that a classical tilting module $T$ over a ring $A$ induces two torsion pairs, one in the category of $A$-modules, the other in the category of modules over the endomorphism ring of $T$, together with a pair of crosswise equivalences between the torsion and torsion-free classes. More precisely, the torsion pair in $\Modr A$ involved in these equivalences is   the torsion pair $(\mathsf{Gen}(T), \Fcal)$ from Theorem~\ref{HKM}, and  in one direction the pair of equivalences is given by the functors $\Hom_A(T,-)$ and $\Ext^1_A(T,-)$, while in the other direction one uses $\otimes$ and Tor.

It turns out that this result extends to silting modules. In \cite{BZ1}, Buan and Zhou establish a Silting Theorem for  a support $\tau$-tilting module $T$ over a finite dimensional algebra $A$. Instead of the endomorphism ring of $T$, however, one needs to take the endomorphism ring $B=\End_{\mathsf D(A)}(\sigma)$ of the corresponding two-term silting complex $\xymatrix{P_{-1} \ar[r]^{\sigma} & P_0}$ in $\mathsf K^b(\mathrm{proj}(A))$.
This is a natural generalisation of  the tilting case, where  $\sigma$, being a projective resolution of $T$, is quasi-isomorphic to $T$, hence $B\cong \End_A(T)$. 

The general case of a silting module $T$ over an arbitrary ring $A$ is treated by Breaz and Modoi in \cite{BMmodule}. Here one has to go a step further and replace $\Modr B$ by  the 
heart $\Hcal_\sigma$ of the silting object $\sigma$. Indeed, in the compact case
 $\Hcal_\sigma\cong \Modr B$ by Proposition~\ref{silting heart}, but in general $\Hcal_\sigma$  just 
 embeds in $\Modr B$. 
 One then obtains a torsion pair $(\Xcal,\Ycal)$ in $\Modr B$ together with
   a pair of crosswise equivalences involving the torsion pair  $(\mathsf{Gen}(T), \Fcal)$ in $\Modr A$  and the subcategories of $\Xcal$ and $\Ycal$ determined by the heart $\Hcal_\sigma$. The pair of equivalences is now given by the functors $\Hom_{\mathsf D(A)}(\sigma,-)$ and $\Hom_{\mathsf D(A)}(\sigma,-[1])$, and also their converses can be described explicitly.

\smallskip

In  representation theory, the Tilting Theorem has  extensively been used to carry knowledge on module categories from one ring to the other, for example, to pass information from the well-understood class of hereditary algebras to the larger class of \emph{tilted} algebras, which are by definition  the endomorphism rings of  classical tilting modules over  finite dimensional hereditary algebras. Tilted algebras have global dimension at most two and satisfy the following homological condition.

\begin{definition}\cite{CL} A finite dimensional algebra $B$ is said to be  \emph{shod (small homological dimension)} if  every indecomposable $B$-module has either projective dimension or injective dimension at most one.
 \end{definition}
 
 In general, the global dimension of a shod algebra is bounded by three. 
 The shod algebras of global dimension at most two are precisely the \emph{quasitilted} algebras, i.e.~the endomorphism rings of tilting objects in Ext-finite hereditary abelian categories, see \cite{HRS}.

 \begin{theorem}\cite{BZ2} The following statements are equivalent for a connected finite dimensional algebra $B$ over an algebraically closed field.
 \begin{enumerate}
\item $B$ is a \emph{silted} algebra, \mbox{i.~e.}~it is the endomorphism ring of  a compact two-term silting complex  over a finite dimensional hereditary algebra;
\item $B$ is a tilted algebra, or it is shod of global dimension three.
\end{enumerate}
\end {theorem}

It is  well known that the global dimension of the endomorphism ring of a tilting module over a ring $A$ can't exceed ${\rm gldim}\,A+1$. We have just seen above that this is no longer true for two-term silting complexes. It is shown in \cite{BZ3} that the endomorphism ring of a compact two-term silting complex over a finite dimensional algebra $A$ has global dimension at most $7$ if  ${\rm gldim}\,A\le 2$, but in general there is no bound. Indeed, for any $n>2$ there is a finite dimensional algebra $A$ of global dimension $n$ with a two-term silting complex whose endomorphism ring has infinite global dimension.

\smallskip

Let us now turn to derived Morita theory.  Happel observed in the 1980's that derived categories are the natural setting for tilting theory. He rephrased the Tilting Theorem by proving that a classical tilting module $T$ over a ring $A$ induces a triangle equivalence $\mathsf D(A)\cong \mathsf D(\End_A(T))$. Later it was shown by Rickard  that two rings $A$ and $B$ are derived equivalent if and only if $B$ is the endomorphism ring of a tilting complex over $A$. 

Over the years, this result has successively been extended to  more general situations \cite{HRS,St,PV,FMS,NSZ,PSt,CHZ,V}, leading to a derived Morita theory for Grothendieck categories. 
Here the role of tilting complexes is  played by exceptional silting objects, that is, silting objects $T$ belonging to $T^{\perp_{\not=0}}$, which we know to coincide with the heart $\Hcal_T$. When $T$ is  not compact, we  have also to take care of coproducts, which leads to the following

\begin{definition} Let $\Acal$ be a Grothendieck category. 

(1) A silting object $T$ in  $\mathrm{D}(\Acal)$  is  \emph{tilting} if $\Add(T)$ is contained in the heart $\Hcal_T$.

(2) A t-structure $(\Vcal,\Wcal)$ in  $\mathrm{D}(\Acal)$   is \emph{intermediate} if 
 there are integers $n\leq m$ such that $\mathsf D^{\leq n}(\Acal)\subseteq \Vcal\subseteq \mathsf D^{\leq m}(\Acal),$
where  $(\mathsf D^{\leq 0}(\Acal), \mathsf D^{>0}(\Acal))$  is the standard t-structure in $\mathrm{D}(\Acal)$.
\end{definition}  


\begin{theorem}\cite[Theorem A]{PV},\cite[Theorem E]{V}
Let $\Acal$ be a Grothendieck category and $\Bcal$ an abelian category.  The following statements are equivalent.
\begin{enumerate}
\item $\Bcal$ has a projective generator, and there is a triangle equivalence $\mathrm{D}(\Bcal)\to\mathrm{D}(\Acal)$ that restricts to bounded derived categories.
\item There is a tilting object $T$ in $\mathrm{D}(\Acal)$ whose heart $\Hcal_T$ is equivalent to $\Bcal$ and whose associated t-structure is intermediate.\end{enumerate}
\end{theorem}

A dual statement holds true for cotilting objects, which we are going to introduce in the next section.

\section{Cosilting objects}\label{cosilt}
Throughout this section, let  $\T$ be a compactly generated triangulated category, and   $A$  a ring with identity. 
We  turn to the dual notions of a cosilting object and a cosilting complex, introduced in \cite{PV}  and \cite{ZW}, respectively. 
\begin{definition} 
(1) An object $C$ in $\Tcal$ is called \emph{cosilting} if $({}^{\perp_{\leq 0}}C,{}^{\perp_{>0}}C)$ is a t-structure in $\T$. 
We call such a t-structure \emph{cosilting}, denote its  heart  by $\Hcal_C$, and write $H^0_C:\T\longrightarrow \Hcal_C$ for the associated cohomological functor. 

(2) A cosilting object $C$ in $\Tcal$  is  \emph{cotilting} if $\Prod(C)$ is contained in the heart $\Hcal_C$.

\smallskip

(3)  A bounded complex of injective $A$-modules $\sigma$ in $\mathsf{K}^b(\mathrm{Inj}(A))$ is a \emph{cosilting complex} if it satisfies the following conditions:
\begin{enumerate}
\item[(C1)] $\Hom_{D(A)}(\sigma^I,\sigma[i])=0$, for all sets $I$ and $i>0$;
\item[(C2)] the smallest triangulated subcategory of $D(A)$ containing $\Prod(\sigma)$ is $\mathsf{K}^b(\mathrm{Inj}(A))$.
\end{enumerate}
\end{definition}

\begin{remark} In \cite{AMV3}, the heart of a cosilting t-structure $(\Ucal, \Vcal)$ is defined as $\Ucal[-1]\cap\Vcal$, deviating by a shift from this note, where we stick to Definition~\ref{def tp}.\end{remark}

We have seen in Theorem~\ref{finite type} that silting complexes inherit a ``finite type'' property from tilting modules. Dually,  cosilting complexes will inherit from cotilting modules the property of being  pure-injective. Let us start by reviewing the relevant concepts.

\subsection{Purity in triangulated categories}
Purity in module categories is a classical subject originating from the theory of abelian groups and extensively studied since the 1950's. The introduction of functorial methods in the theory of purity goes back to the 1970's. Following the approach developed by Gruson and Jensen, one embeds $\Modr A$ into the category $(A\lmod, \Ab)$ of all  covariant additive functors from finitely presented left $A$-modules to abelian groups, replacing a module $M$ by the functor $(M\otimes_A -)\mid_{A\lmod}$. Pure-exact sequences of modules then correspond to short exact sequences in the functor category $(A\lmod, \Ab)$, and pure-injective modules to injective objects. Questions on  pure-injective modules are thus translated into questions on injectives in the locally coherent Grothendieck category $(A\lmod, \Ab)$, for which one can employ powerful tools such as localisation techniques and dimension theory. 

The theory of purity in a compactly generated triangulated category $\Tcal$ has been developed around 2000 in \cite{Kr1,Bel}. The approach is essentially the same,  the role of the finitely presented modules  now being played by the compact objects. More precisely, one associates to $\T$  the locally coherent Grothendieck
 category $\Modr\Tcal^c$ of all contravariant additive functors $\T^c\rightarrow \Ab$ from compact objects  to abelian groups,  via 
 the {restricted Yoneda functor} $\mathbf{y}:\T\to\Modr\T^c$
which assigns to an object $X$  the functor $\mathbf{y}X=\Hom_\Tcal(-,X)_{|\Tcal^c}$. 
 
\begin{definition} A triangle 
$\Delta: \xymatrix{X\ar[r]^f&Y\ar[r]^g&Z\ar[r]&X[1]}$ 
in $\T$ is said to be a \textit{pure triangle} if $\mathbf{y}\Delta$ is a short exact sequence in  $\Modr\Tcal^c$, that is, for any compact object $K$ in $\T$, the sequence
$$\xymatrix{0\ar[r]&\Hom_{\T}(K,X)\ar[rr]^{\Hom_\T(K,f)}&&\Hom_\T(K,Y)\ar[rr]^{\Hom_\T(K,g)}&&\Hom_\T(K,Z)\ar[r]&0}$$
is exact. We then refer to $f:X\rightarrow Y$  as a \textit{pure monomorphism}.  
The objects $E$ of $\T$ for which every pure monomorphism of the form $f:E\rightarrow Y$ in $\Tcal$ splits are  called \textit{pure-injective}. In other words, an object $E$   is pure-injective if and only if  $\mathbf{y}E$ is an injective object in $\Modr\T^c$. 
\end{definition} 


Notice that, unlike the abelian case,  $\mathbf y$ is not a fully faithful embedding in general.
In fact, the  objects $E$ for which $\mathbf{y}$ induces an isomorphism $\Hom_\Tcal(X,E)\cong \Hom_{\Modr\Tcal^c}(\mathbf{y}{X},\mathbf{y}{E})$ for any  $X$ in $\Tcal$ are precisely the pure-injective ones. For details, we refer to \cite{Kr1,Bel,Prest}.

\smallskip

We are going to see that pure-injectivity of an object can often be phrased in terms of the notion of a definable subcategory. This is a concept which   has its origins in the model theoretic approach to representation theory. Definable subcategories of  $\Modr A$ are the zero sets of pp-functors. The latter are precisely the functors $F:\Modr A\to \Ab$ which commute with direct limits and products, or equivalently, which admit a presentation $\Hom_A(K,-)\longrightarrow \Hom_A(L,-)\longrightarrow F\rightarrow 0$ with   $K,L\in\modr A$. The triangulated version reads as follows.
\begin{definition}\cite{Kr2} (1) A covariant additive functor $F:\Tcal\longrightarrow \Ab$ is said to be \textit{coherent} if it admits a presentation 
$\Hom_\T(K,-)\longrightarrow \Hom_\T(L,-)\longrightarrow F\rightarrow 0$
with  $K,L\in\T^c$, or equivalently,  $F$ preserves products and coproducts and takes pure triangles to short exact sequences. 

\smallskip

(2)
A subcategory $\Vcal$ of $\Tcal$ is said to be \textit{definable} if there is a family of coherent functors \mbox{$F_i: \T\longrightarrow\Ab,\, {i\in I},$} 
such that $X$ lies in $\Vcal$ if and only if $F_i(X)=0$ for all $i$ in $I$.
\end{definition}

\begin{examples}\label{defble}
(1) If $\Scal$ is a set of compact objects in $\Tcal$, then $\Vcal=\Scal^{\perp_0}$ is definable.

(2) \cite[Lemma 4.8]{AMV3} 
If  $\mathscr{C}$ is an additive subcategory of $\Tcal$, then $\Vcal={}^{\perp_{>0}}\mathscr{C}$ is definable if and only if it  is closed under products and every object in $\mathscr{C}$ is pure-injective.
\end{examples}

By definition,  definable subcategories of $\T$ are closed under products, coproducts, pure subobjects and pure quotients. 
By \cite[Corollary 4.4]{AMV3} 
they are also closed under pure-injective envelopes.
The same closure conditions are verified by definable subcategories of module categories. In fact, closure under direct products, pure submodules and direct limits even characterises definability in $\Modr A$. We are going to see that a  similar result can be achieved in   $\T$ when there is  an enhancement provided by a strong and stable derivator $\mathbb D$. 

Roughly speaking, one assumes the existence of a 2-functor 
$\mathbb D:Cat\,^{op}\longrightarrow CAT$ from the 2-category $Cat$ of all small categories to the 2-category $CAT$ of all categories, which satisfies certain axioms and has the property that $\T$ equals $\mathbb D(\mathds{1})$, where $\mathds{1}$ is the category with a single object and its identity morphism. We will say that $\T$ is the \textit{underlying category}, or the  base,  of the derivator  $\mathbb D$. We refer to  \cite{SSV} and \cite{L} for details. Here we only remark that all compactly generated triangulated categories which are algebraic, and in particular, all derived module categories, admit such an enhancement.

The axioms on $\mathbb D$ ensure that for any small category $I$ the unique functor $\pi:I\to \mathds 1$ gives rise to a functor $\pi^\ast: \mathbb D(\mathds 1)\to \mathbb D(I)$ together with a left adjoint $\pi_!$ and a right adjoint $\pi_\ast$. One refers to $\pi_!$ as the \textit{homotopy colimit functor} 
$\hocolim_I:\mathbb D(I)\to \mathbb D(\mathds 1)$ and to $\pi_\ast$ as the \textit{homotopy limit functor} 
$\holim_I:\mathbb D(I)\to \mathbb D(\mathds 1)$. 

Note that every object $i$ in $I$ defines a unique functor $i:\mathds 1 \to I$, which by the axioms on $\mathbb D$ induces a functor $i^\ast:\mathbb D(I)\to \mathbb D(\mathds 1)$.  The image of an   object $X$ in $\mathbb  D(I)$ under this functor is denoted by  $X_i=i^\ast(X)$.
We will say that a class of objects $\Vcal$ in $ D(\mathds 1)$ is \textit{closed under directed homotopy colimits} if  for all directed categories $I$ and all objects $X$ in  $\mathbb D(I)$ having the property that $X_i$ lies in $\Vcal$ for all $i$ in $I$ we have that the homotopy colimit $\hocolim_I(X)$ belongs to $\Vcal$. 

 Now we can state the characterisation of definability announced above.

\begin{theorem}\cite[Theorem 3.11]{L}\label{closure} Assume that $\T$ is a compactly generated triangulated category which  is the underlying category of a strong and stable derivator. A subcategory  of $\T$ is definable if and only if it  is closed under 
products, pure subobjects and  directed homotopy colimits.
\end{theorem}

Finally, let us point out that definable subcategories are  ``functorially finite'', in the sense that they provide both left and right approximations. This is a well-known result for module categories 
which has the following triangulated version.

\begin{theorem}\label{ff} Let   $\Vcal$ be a definable subcategory of  a compactly generated triangulated category $\T$.
\begin{enumerate}
\item \cite[Proposition 4.5]{AMV3} $\Vcal$ is preenveloping. 
\item \cite{LV} Assume that $\T$ is algebraic. Then $\Vcal$ is  precovering. Moreover, if $\Vcal$ is closed under extensions, then it is the aisle of a torsion pair.
\end{enumerate}
\end{theorem}

\subsection{Pure-injective cosilting objects} 
We now return to cosilting complexes.  
 It is shown in \cite{MV}, dually to Theorem~\ref{finite type}, that every cosilting complex  can be interpreted as a cotilting module in a suitable module category.  Observe that both in module categories and in triangulated categories, pure-injectivity of an object $E$ can be characterised by the following factorisation property: For every set $I$, the summation map $E^{(I)}\rightarrow E$ factors through the canonical map $E^{(I)}\rightarrow E^I$. This allows to translate to $\mathsf D(A)$ another important result from tilting theory 
 stating that every cotilting module is pure-injective.
\begin{proposition}\cite[Proposition 3.10]{MV}\label{FJ} (1) Every (bounded) cosilting complex is pure injective.

(2) A bounded complex of injective $A$-modules is a cosilting object in $\mathsf{D}(A)$ if and only if it is a {cosilting complex}. \end{proposition}

Let us sketch the argument for the second statement. Given a cosilting complex $C$, one uses that $\Vcal={}^{\perp_{>0}}C$ is closed under products by \cite[Proposition  2.12]{ZW} and therefore definable by Example~\ref{defble}(2). Combining Theorems~\ref{ff} and \ref{N}, it follows that $\Vcal$ is the coaisle of a t-structure, and one checks that the aisle equals 
${}^{\perp_{\leq 0}}C$. The proof of the only-if-part  goes back to \cite[Proposition 2.10]{ZW} and
 is dual to the proof of Proposition~\ref{obj=cpx}.

\smallskip


Assume now that $C$ is a  cosilting object in $\Tcal$ with associated t-structure $\mathfrak{t}=({}^{\perp_{\leq 0}}C,{}^{\perp_{>0}}C)$. If $C$ is pure-injective, then  $\mathbf{y}{C}$ is an injective object in the functor category $\Modr\Tcal^c$. Following the strategy in \cite{St}, 
we consider the
hereditary torsion pair  $({}^{\perp_0}\mathbf{y}C,\mathrm{Cogen}(\mathbf{y}{C}))$ in $\Modr\Tcal^c$ which is  cogenerated by $\mathbf{y}{C}$. 
The quotient category $$\mathcal{G}_C:=\Modr\Tcal^c/{}^{\perp_0}\mathbf{y}C$$ is a Grothendieck category (see \cite[Proposition III.9]{Gabriel}). Its category of injective objects is equivalent to $\Prod(\mathbf{y}C)$.

On the other hand, dually to Proposition~\ref{silting heart}(1),
one verifies that the cohomological functor $H^0_C:\T\longrightarrow \mathcal H_C$ maps $C$ to an injective cogenerator of the heart, hence the category of injective objects in $\Hcal_C$ is $\Prod(H^0_C(C))$.  
Moreover, one can construct a commutative diagram of functors
$$\xymatrix{\T\ar[dr]_{{H^0_C}}\ar[r]^{\mathbf{y}\;}&{\Modr\Tcal^c}\ar[r]^{\;\pi}&\mathcal{G}_C\ar[dl]^{F}\\
&\mathcal H_C  &}$$
where  $\pi:\Modr\Tcal^c\longrightarrow \Gcal_C$ is the quotient functor, and $F$ is a left exact functor between the abelian categories $\mathcal{G}_C$ and $\Hcal_C$. Since $F$ restricts to  an equivalence on the categories of injective objects, it follows that it  is an equivalence   $\mathcal{G}_C\cong\Hcal_C$. In particular, $\Hcal_C$ is a Grothendieck category.

Dualising the  first bijection in Theorem~\ref{NSZ-bijection} we obtain the following result.
\begin{theorem}\label{NSZ-cobijection}\cite[Theorem 3.5 and Corollary 3.8]{AMV3}
If $\T$ is a compactly generated triangulated category, there is a   bijective correspondence between 
\begin{enumerate}
\item[(i)] equivalence classes of cosilting objects;
\item[(ii)] nondegenerate smashing t-structures  whose heart admits a injective cogenerator;
\end{enumerate}
which restricts to a bijection between
\begin{enumerate}
\item[(i')] equivalence classes of pure-injective cosilting objects;
\item[(ii')] nondegenerate smashing t-structures  whose heart is a Grothendieck category.
\end{enumerate}
\end{theorem}

The question of finding conditions ensuring that the heart of a t-structure is a Grothendieck category has received a lot of attention in recent years, see e.g.~\cite{PaSa,SSV}. The following condition plays an important role in this context.

\begin{definition} A t-structure $\mathfrak{t}=(\Ucal,\Vcal)$,  or a cosuspended TTF triple $(\Ucal,\Vcal,\Wcal)$  in $\T$ are said to be \emph{homotopically smashing} if the class $\Vcal$ is closed under directed homotopy colimits.   \end{definition}

As shown in \cite[Proposition 5.6]{SSV}, the  condition ``homotopically smashing'' on a t-structure $\mathfrak{t}$ is weaker than ``compactly generated'' (to see this one can also combine Example~\ref{defble}(1)  with Theorem~\ref{closure}), but stronger than ``smashing''.
In presence of an enhancement of $\T$ by a strong and stable derivator, it implies that the heart  $\Hcal_{\mathfrak{t}}$ has exact directed colimits, see \cite[Theorem B]{SSV}. Moreover, it is proved in \cite[Lemma 4.3]{L} that the cohomological functor $H^0_{\mathfrak{t}}$ maps a set of representatives  of $\T^c$ to a generating set of the heart, hence $\Hcal_{\mathfrak{t}}$ is a Grothendieck category. Combining this with  Example~\ref{defble}(2) and Theorems~\ref{closure} and~\ref{NSZ-cobijection} one obtains the following result. 

\begin{theorem}\cite[Theorem 4.6]{L}\label{R} Assume that $\T$ is a compactly generated triangulated category which  is the underlying category of a strong and stable derivator. The following statements are equivalent for a nondegenerate  t-structure $\mathfrak{t}=(\Ucal, \Vcal)$.
\begin{enumerate}
\item $\mathfrak{t}$ is a cosilting t-structure given by  a pure-injective cosilting object.
\item $\Vcal$ is a definable subcategory of $\Tcal$.
\item $\mathfrak{t}$ is {homotopically smashing}.
\item  $\mathfrak{t}$ is {smashing} and its heart is a Grothendieck category.
\end{enumerate}
\end{theorem}

In particular, the theorem applies to  nondegenerate compactly generated t-structures.

\begin{corollary}\label{cg} Assume that $\T$ is a compactly generated triangulated category which  is the underlying category of a strong and stable derivator. Then every compactly generated, nondegenerate  t-structure $\mathfrak{t}$  in $\T$  is a cosilting t-structure given by  a pure-injective cosilting module, and the heart of $\mathfrak{t}$  is a Grothendieck category.
\end{corollary}

Further results showing that hearts of compactly generated t-structures are usually Grothendieck categories are given in \cite[Theorem 5.4.2]{Bondarko2016} and \cite[Corollary D]{SSV}.
An example of a t-structure as in Theorem~\ref{R} which is not compactly generated will be given below in Example~\ref{notcoft}. 

\smallskip

If $\T$ is algebraic,  every t-structure as in Theorem~\ref{R} has a right-adjacent co-t-structure by  Theorem~\ref{ff}. One can then dualise the  second bijection in Theorem~\ref{NSZ-bijection} as follows.

\begin{theorem}\label{NSZ-cobijection2} If $\T$ is an algebraic compactly generated triangulated category, 
there is a   bijective correspondence between 
\begin{enumerate}
\item[(i)] equivalence classes of pure-injective cosilting objects;
\item[(ii)] nondegenerate cosuspended TTF-triples $(\Ucal,\Vcal,\Wcal)$ which are homotopically smashing. \end{enumerate}
\end{theorem}


For the  dual version
 of Theorem~\ref{bijections general} we need the following terminology. 

\begin{definition}
A co-t-structure $(\Vcal,\Wcal)$, or a cosuspended TTF triple $(\Ucal,\Vcal,\Wcal)$ 
in $\mathsf D(A)$ are said to be \emph{cointermediate} if there are integers $n\geq m$ such that $\mathsf D^{\geq n}\subseteq \Vcal\subseteq \mathsf D^{\geq m}.$
\end{definition}

Notice that a cosuspended TTF triple $(\Ucal,\Vcal,\Wcal)$ 
is {cointermediate} if and only if $(\Ucal,\Vcal)$ is an intermediate t-structure. Similarly,
a suspended TTF triple $(\Ucal,\Vcal,\Wcal)$ 
is {intermediate} if and only if $(\Ucal,\Vcal)$ is an cointermediate co-t-structure.

\begin{theorem}\cite[Theorem 3.13]{MV}\label{bijections cogeneral} There is a bijection between
\begin{enumerate}
\item[(i)] equivalence classes of (bounded) cosilting complexes in $\mathsf D(A)$;
\item[(ii)] cointermediate cosuspended TTF triples in  $\mathsf D(A)$.\end{enumerate}
\end{theorem}
Finally, there is also a dual version of Theorem~\ref{HKM} describing the zero-cohomologies of two-term cosilting complexes. This gives rise to the notion of a  \emph{cosilting module} first introduced in \cite{BPop}. Cosilting modules parametrise the torsion pairs  in $\Modr A$ whose torsion-free class is definable. For details we refer to \cite{BZ,ZW,abundance}.

\begin{example}\cite[Example 5.4]{AH1}\cite[Example 3.12]{MV} \cite[Example 4.10]{abundance}\label{notcoft}
Let $(A,\mathfrak m)$ be a valuation domain whose maximal ideal $\mathfrak m=\mathfrak m^2$ is idempotent and non-zero.
Then $S=A/\mathfrak m$  is a  cosilting module cogenerating the definable torsion-free class $\Add(S)=\{M\in\Modr A\,\mid\,M\mathfrak m=0\}$. Notice that the  torsion pair  $(\Ker\Hom_A(-,S),\Add(S))$  is not hereditary, because
 $\Add(S)$ does not contain the injective envelope of $S$. From \cite[Lemma 3.7 and 4.2]{AH1} we infer that there is no set of finitely presented $A$-modules generating this torsion pair.
Then it follows from \cite[Theorem 2.3]{BP} that its HRS-tilt $\mathfrak{t}=(\Ucal,\Vcal)$ can't be compactly generated. On the other hand,  as a  cosilting module, $S$ is associated to a two-term cosilting complex whose corresponding t-structure is precisely $\mathfrak{t}$. 
Hence 
$\mathfrak{t}$ is a cosilting t-structure which is not compactly generated and whose heart is a Grothendieck category.
\end{example}

\section{Classification results}
In this section we  present classification results for silting and cosilting complexes. We deal with two cases: commutative noetherian rings, and finite-dimensional algebras.  In the first case, we have the same phenomenon described in \cite{abundance} for modules:  cosilting objects turn out to be more accessible, and the classification of silting objects is then deduced via duality. 

In the second case, we focus on silting-discrete algebras,   the class of finite-dimensional algebras for which all silting complexes are compact and their number is finite, up to equivalence and shift. We discuss general properties of such algebras and  review some classification results.

\subsection{Silting-cosilting duality}
Let $A$ be an arbitrary ring. Let  $k$ be
 a commutative ring such that  $A$ is a $k$-algebra, and  $W$  an injective cogenerator in $\Modr k$.  For example, one can choose $k=\mathbb Z$ and  $W= \mathbb{Q}/\mathbb{Z}$.
 
 We will use the duality functors  $(-)^+ = \RHom_k(-,W)$
between the derived categories $\mathsf D(A)$ and $\mathsf D(A^{op})$ of right and left $A$-modules. 
We will also need $(-)^*:=\RHom_A(-,A)$, considered as a functor  both on $\mathsf D(A)$ and $\mathsf D(A^{op})$, depending on the context. 
Recall that $(-)^*$ induces a duality between $\mathrm{proj}(A)$ and  $\mathrm{proj}(A^{op})$ which extends, via d\' evissage, to a duality between $\mathsf{K}^b(\mathrm{proj}(A))$ and  $K^b(\mathrm{proj}(A^{op}))$, or in other terms,  between $\mathsf D(A)^c$ and $\mathsf D(A^{op})^c$.

The following formula proved in \cite{AH2} relates the two duality functors considered above:
\begin{equation}\label{elementarydual} \Hom_{\mathsf D(A)}(S,X)\cong\Hom_{\mathsf D(A^{op})}(S^*,X^+)\:\text{ for all }\:
S\in \mathsf{K}^b(\mathrm{proj}(A))\text{ and }X\in\mathsf D(A)\end{equation}
Thus, if
$\Scal$ is a set of objects in $\mathsf{K}^b(\mathrm{proj}(A))$, and $\Scal^*$ is the set in $K^b(\mathrm{proj}(A^{op}))$ given by the objects $S^*$ with $S\in\Scal$, then an object $X$ in ${\mathsf D(A)}$ belongs to the orthogonal category $\Scal^{\perp_0}$ if and only if $X^+$ belongs to $(\Scal^*)^{\perp_0}$. 

Following an idea from \cite[Theorem 4.10]{StPo}, we consider the following assignment.  
\begin{theorem}\label{T:torsionpairduality}\cite{AH2}
	The map $\Psi$ which assigns to a torsion pair   in 
$\mathsf D(A)$ generated by a set of compact objects $\Scal$ the torsion pair in  $\mathsf D(A^{op})$ generated by $\Scal^*$
defines  a bijection between
	\begin{enumerate}
\item[(i)] compactly generated torsion pairs in   $\mathsf D(A)$,
\item[(ii)] compactly generated torsion pairs in  $\mathsf D(A^{op})$,
\end{enumerate}
which restricts to a bijection  between
	\begin{enumerate}
\item[(i)] compactly generated t-structures in   $\mathsf D(A)$,
\item[(ii)] compactly generated co-t-structures in  $\mathsf D(A^{op})$,
\end{enumerate}
and maps intermediate t-structures to cointermediate co-t-structures, and vice versa.
\end{theorem}

\begin{definition}
	A (co)silting object  in a compactly generated triangulated category $\T$ is of \emph{(co)finite type} provided that the (co)aisle of the corresponding (co)silting t-structure is of the form $\Scal^{\perp_0}$ for a set of compact objects $\Scal$.
\end{definition}

Let us focus on the case $\T=\mathsf D(A)$. We know from Theorem~\ref{finite type} and Example~\ref{notcoft} that all(bounded) silting complexes   are of finite type, while cosilting complexes need not be of cofinite type. On the other hand, by Corollary~\ref{cg}, every nondegenerate compactly generated t-structure is   a cosilting t-structure given by a pure-injective cosilting object of cofinite type.

\begin{theorem}\cite{AH2} \label{T:siltingduality}
There is a bijection between
\begin{enumerate}
\item[(i)] equivalence classes of (bounded) silting complexes in   $\mathsf D(A)$,
\item[(ii)] equivalence classes of (bounded) cosilting complexes of cofinite type in  $\mathsf D(A^{op})$,
\end{enumerate}
which, up to equivalence, maps  a silting complex $\sigma$ to the cosilting complex  $\sigma^+$. 
\end{theorem}
\begin{proof}
We apply Theorem~\ref{T:torsionpairduality}. Consider
 an intermediate suspended TTF $(\Ucal,\Vcal,\Wcal)$ in   $\mathsf D(A)$, which is  compactly generated by Theorem~\ref{finite type}. The compactly generated cointermediate co-t-structure $(\Ucal,\Vcal)$ is  mapped by $\Psi$    to a compactly generated  intermediate t-structure $(\Ucal',\Vcal')$.
Now Theorem~\ref{coBondarko}  yields the existence of a right adjacent co-t-structure. So, we obtain a compactly generated cosuspended cointermediate TTF $(\Ucal',\Vcal',\Wcal')$.
Conversely, if we start with a compactly generated cosuspended cointermediate TTF $(\Ucal',\Vcal',\Wcal')$ in   $\mathsf D(A^{op})$, we know that the compactly generated intermediate t-structure  $(\Ucal',\Vcal')$ is the image under $\Psi$ of a compactly generated cointermediate co-t-structure, which has a right-adjacent t-structure by Theorem~\ref{coBondarko}, yielding  a compactly generated suspended intermediate TTF $(\Ucal,\Vcal,\Wcal)$.  
  The stated bijection  now follows from Theorems~\ref{bijections general} and~\ref{bijections cogeneral}.
  
  Moreover, we know from formula (\ref{elementarydual}) that  an object $X\in\mathsf D(A)$ belongs to $\Vcal'$ if and only if its dual $X^+$ belongs to the silting class $\Vcal=\sigma^{\perp_{>0}}$. Now one checks that the latter amounts to $X\in{}^{\perp_{>0}}C$ for $C=\sigma^+$. Hence $\Vcal'={}^{\perp_{>0}}C$ and standard arguments show 
  that $\Ucal'={}^{\perp_{\le 0}}C$.
  \end{proof}

\begin{remark} \label{injection}
Similarly, one can use $\Psi$ to  map equivalence classes of silting objects of finite type in   $\mathsf D(A)$ to
 equivalence classes of (pure-injective) cosilting objects of cofinite type in  $\mathsf D(A^{op})$. Indeed, one can show that the t-structure  $(\Ucal',\Vcal')$  is nondegenerate if so is the suspended TTF $(\Ucal,\Vcal,\Wcal)$. It then follows from Corollary~\ref{cg} that the assignment is well-defined  and obviously injective.  We will see that it is also surjective when the ring $A$ is commutative noetherian, cf.~Theorem~\ref{commnoeth}.
 \end{remark}

\subsection{Commutative noetherian rings.} 
Over a commutative noetherian ring, various types of torsion pairs, both in the module category and in the derived category, can be classified geometrically in terms of subsets of the prime spectrum. We will now employ such results to obtain  parametrisations of silting and cosilting objects.

 Throughout this subsection, $A$ will be a commutative noetherian ring, and $\Spec A$ will denote the prime spectrum of $A$ endowed with the Zariski topology, where the closed sets are the sets of the form
$ V(I) = \{ \p\in\Spec A \mid \p \supseteq I \}$ for some ideal $I \subseteq A$.
For $\p \in \Spec A$, we denote by $A_\p$ the localisation of $A$ at $\p$.
Given a module $M$, we denote by $\Supp\, M=\{\p\in\Spec A\mid M\otimes_A A_\p\not=0\}$ the \emph{support} of $M$.

A subset $V \subseteq \Spec A$ is said to be \emph{closed under specialisation}  if it contains $V(\p)$ for any $\p \in V$. 
Work of  Gabriel \cite{Gabriel} establishes a bijective correspondence between the subsets of $\Spec{A}$ closed under specialisation and the hereditary torsion pairs in $\Modr A$ via the assignment of support.
More precisely, given a hereditary torsion pair $(\mathcal X,\mathcal Y)$,  it is easy to see that the union of all supports of modules in $\Xcal$ forms a specialisation-closed subset of $\Spec A$. Conversely, 
every specialisation-closed subset $V \subseteq \Spec{A}$   determines a hereditary torsion pair $(\mathcal X_V,\mathcal Y_V)$ where the torsion class $\Xcal_V=\{ M\in \Modr A \mid \Supp{M}\subseteq V\}$ consists of the modules supported in $V$.

It is shown in \cite[Theorem 5.1]{AH1} that the hereditary torsion pairs in $\Modr A$ are precisely the  torsion pairs induced by two-term cosilting complexes, and moreover, the latter are all of cofinite type. This yields a parametrisation of silting and cosilting two-term complexes by specialisation-closed subsets  of $\Spec A$. Let us now turn to the general case.


\begin{definition}  A \emph{filtration by supports} of $\Spec
A$ is a  map $\Phi :\mathbb Z\longrightarrow\mathcal{P}(\Spec A)$ such
that each $\Phi (n)$ is a subset of $\Spec A$ closed under
specialisation and  $\Phi (n)\supseteq\Phi (n+1)$ for all
$n\in\mathbb Z$.

Furthermore, we will say that {filtration by supports} $\Phi$ is 
\begin{itemize}
\item \emph{intermediate} if there are   integers $n\leq m$ such that $\Phi(n)=\Spec A$  and $\Phi(m)=\emptyset$; 
\item \emph{nondegenerate} if $\bigcup_{n\in\mathbb Z} \Phi(n)=\Spec A$  and  $\bigcap_{n\in\mathbb Z} \Phi(n)=\emptyset$. 
\end{itemize}
\end{definition}

Observe that every filtration by supports  corresponds to a family of hereditary torsion pairs $(\Xcal_n,\Ycal_n)_{n\in\mathbb Z}$ given by a decreasing sequence of torsion classes $\ldots\supseteq\Xcal_n\supseteq\Xcal_{n+1}\supseteq\ldots$ where each $\Xcal_n$ consists of the modules supported in $\Phi(n)$.

\begin{theorem}\cite[Theorem 3.11]{AJSa} \label{thm:AJS}
Let $A$ be a commutative noetherian ring. There is a bijective correspondence between
\begin{enumerate}
\item[(i)] compactly generated t-structures in $\mathsf D(A)$;
\item[(ii)] filtrations by supports of $\Spec A$.
\end{enumerate}
The correspondence from (i) to (ii) is given by
$(\mathcal{U},\Vcal)\mapsto \Phi_\mathcal{U}$, where
$$\Phi_\mathcal{U}(n)=\{\p\in\Spec A\text{:
}A/\p[-n]\in\mathcal{U}\},$$ while the correspondence from (ii) to
(i) maps $\Phi$ to the t-structure $(\mathcal{U}_\Phi
,\mathcal{V}_\Phi)$ where
$$\mathcal{U}_\Phi 
=\{X\in \mathsf D(A)\text{: }\Supp\, H^n(X)\subseteq\Phi (n)\text{ for all }n\in\mathbb Z\}.$$
\end{theorem}

\begin{remark}\label{filtra} (1) A compactly generated t-structure $(\Ucal_\Phi, \Vcal_\Phi)$ is intermediate if and only if so is the associated filtration by supports $\Phi$. Indeed, the existence of   integers $n\leq m$ such that $\mathsf D^{\leq n}\subseteq \Ucal_{\Phi}\subseteq \mathsf D^{\leq m}$ means precisely that the  cohomologies of objects in $\Ucal_{\Phi}$ are arbitrary in degrees  $\le n$ and vanish in degrees $>m$. In other words, $\Phi(i)=\Spec A$ for all $i\le n$ and $\Phi(i)=\emptyset$ for all $i> m$.

\medskip 

(2) Similarly,  $(\Ucal_\Phi, \Vcal_\Phi)$ is nondegenerate if and only if so is the filtration by supports  $\Phi$.
 The  condition on the intersection of the $\Phi(n)$ follows immediately from the fact that $\bigcap_{n\in\mathbb Z}\Ucal_\Phi[n]$ consists of the objects $X\in\mathsf D(A)$
 whose cohomologies are supported in $\bigcap_{n\in\mathbb Z} \Phi(n)$.
Moreover, if $\Scal$ is a set of compact objects generating $(\Ucal_\Phi, \Vcal_\Phi)$, then the condition
$\bigcap_{n\in\mathbb Z}\Vcal_\Phi[n]=0$ holds if and only if $\Scal^{\perp_{\mathbb Z}}=0$, which  amounts to $\mathsf{Loc}(\Scal)=\mathsf D(A)$ by Theorem~\ref{CGR}. 
Using the one-one correspondence  established by Hopkins and Neeman \cite{chrom} between localising subcategories of $\mathsf D(A)$ and subsets of $\Spec A$ via the assignment of support, one concludes that
the latter condition
is equivalent to $\bigcup_{n\in\mathbb Z} \Phi(n)=\Spec A$.\end{remark}

It is shown in \cite{Hrbek}  that over a commutative noetherian ring, every homotopically smashing  t-structure $(\Ucal,\Vcal)$ with $\Vcal\subseteq \mathsf D^{\geq m}$ for some $m\in\mathbb Z$ is compactly generated. 
Combining this with the results in Section~\ref{cosilt}, one obtains that  all cosilting complexes are of cofinite type. In fact, one can prove the following result which extends work from \cite{ASao,AH1}.

\begin{theorem}\cite{Hrbek,AH2}\label{commnoeth}
	Let $A$ be a commutative noetherian ring. There is a bijective correspondence between		
	\begin{enumerate}
	\item[(i)] equivalence classes of silting objects of finite type,
	\item[(ii)] equivalence classes of cosilting objects of cofinite type,
	\item[(iii)] 	nondegenerate	filtrations by supports,		\end{enumerate}
	which restricts to a bijective correspondence between		
	\begin{enumerate}
				\item[(i')] equivalence classes of (bounded) silting complexes,
			\item[(ii')] equivalence classes of (bounded) cosilting complexes,
	\item[(iii')] 	intermediate filtrations by supports.		
	\end{enumerate}
In particular, all (bounded) cosilting complexes are of cofinite type.
\end{theorem}
\begin{proof} The bijections (ii)$\leftrightarrow$(iii) 
and (ii')$\leftrightarrow$(iii') 
follow from Remark~\ref{filtra},  and Theorems~\ref{NSZ-cobijection2} and~\ref{bijections cogeneral}.
Moreover, we have seen in Remark~\ref{injection} that there is always an injective map (i)$\to$(ii). To prove the surjectivity, we consider a compactly generated nondegenerate t-structure $(\Ucal',\Vcal')$, which we know to correspond under $\Psi$ to a compactly generated suspended TTF triple $(\Ucal,\Vcal,\Wcal)$. Using that the filtration by supports  corresponding to $(\Ucal',\Vcal')$ is nondegenerate, one can prove that for any $\p\in\Spec A$ there is an integer $n$ such that $R_\p[n]$ lies in  $\Vcal$. Then $\mathsf{Loc}(\Vcal)$ contains the stalk complex $R_\p$ for each $ \p\in\Spec A$. Like  in Remark~\ref{filtra}(2), one can use that localising subcategories of $\mathsf D(A)$ are in bijection with  subsets of $\Spec A$ via the assignment of support to conclude that $\bigcap_{n\in\mathbb Z}\Wcal[n]=0$. It follows that
the suspended TTF triple $(\Ucal,\Vcal,\Wcal)$ is nondegenerate and thus corresponds to a silting object.
 \end{proof}
 
 We can now pin down an example of an unbounded (co)silting object of (co)finite type in $\mathsf{D}(A)$. All we need to do is  to find a filtration by supports  which is nondegenerate, but not intermediate. 
 
\begin{example}\cite{AH2}\label{unbounded}
	Let $\mathbb{P} = \{p_1,p_2,p_3,\ldots\}$ be an ordering of all prime numbers. We consider the following filtration by supports $\Phi$ in $\Spec(\mathbb{Z})$:
	$$\Phi(n) = \begin{cases} \Spec(\mathbb{Z}) & \text{for $n < 0$}, \\ \mathbb{P} \setminus \{p_1,p_2,\ldots,p_n\} & \text{for $n \geq 0$}. \end{cases}.$$
		Under the bijections of Theorem~\ref{commnoeth}, the filtration by supports $\Phi$ corresponds to the (unbounded) cosilting   object
	$$C = \bigoplus_{n \geq 0} \mathbb{J}_{p_{n+1}}[-n] \oplus \mathbb{Q}$$
	and to  the (unbounded) silting object
	$$T = \bigoplus_{n \geq 0} \mathbb{Z}^\infty_{p_{n+1}}[n] \oplus \mathbb{Q}$$
	where $\mathbb{J}_p$  and  $\mathbb{Z}^\infty_p$ denote the group of $p$-adic integers and the $p$-Pr\"ufer group corresponding to a prime $p$, respectively. 
	\end{example}	
	
	Similar examples can be constructed  over finite dimensional hereditary algebras, e.g.~over the path algebra of the Kronecker quiver $\xy\xymatrixcolsep{2pc}\xymatrix{ \bullet \ar@<0.5ex>[r]  \ar@<-0.5ex>[r] & \bullet } \endxy$.

 \subsection{Silting-discrete algebras}\label{discrete}
We now turn to a class of algebras where classification is possible as   there are only finitely many silting complexes up to equivalence and shift. This is phrased in terms of ``local finiteness'' of the poset of basic silting objects, that is,  any interval of the poset contains only finitely many objects. 

\smallskip

Throughout this section, we denote by $A$ a finite dimensional algebra over a field $k$ and set $\T=\mathsf{K}^b(\mathrm{proj}(A))$. Recall that the silting subcategories of $\Tcal$ are parametrised by the compact basic silting complexes in $\mathsf{D}(A)$, see Remark~\ref{commentstodef}(2). The poset of silting subcategories 
 $(\mathsf{silt}\,\T, \ge)$ can then be regarded as poset of basic silting complexes, where  $$T_1\ge T_2\quad\text{ if and only if  }\quad T_1^{\perp_{>0}}\ge T_2^{\perp_{>0}}.$$
\begin{definition}
A finite dimensional algebra $A$ is said to be {\it silting-discrete} if for any two objects 
$M_1,M_2$ in $\mathsf{silt}\,\T$ with $M_1\ge M_2$ there are only finitely many objects $T$ in $\mathsf{silt}\,\T$ with $M_1\ge T\ge M_2$. 
\end{definition}

We collect some  handy characterisations of silting-discrete algebras below. In particular, we are going to see that the definition can be rephrased as follows: for any $n$ there are only finitely many intermediate t-structures $(\Vcal,\Wcal)$ in $\mathsf{D}(A)$ with $\mathsf D^{\leq -n}\subset \Vcal\subset\mathsf D^{\leq 0}.$

\begin{theorem}\cite[Proposition 3.8]{Aihara},\cite[Theorem 2.4]{AM} \cite[Proposition 3.27]{AMY}
The following statements are equivalent for a finite dimensional algebra $A$.
\begin{enumerate}
\item $A$ is silting-discrete.
\item There is an object $M$  in $\mathsf{silt}\,\T$ such that for each integer $k>0$ there are only finitely many objects $T$ in $\mathsf{silt}\,\T$ satisfying $M\ge T\ge M[k]$.
\item For any object $M$  in $\mathsf{silt}\,\T$ and each integer $k>0$ there are only finitely many objects $T$ in $\mathsf{silt}\,\T$ with $M\ge T\ge M[k]$.
\item For any  $M$  in $\mathsf{silt}\,\T$ there are only finitely many objects $T$ in $\mathsf{silt}\,\T$ with $M\ge T\ge M[1]$.
\item For any $M$ in $\mathsf{silt}\,\T$ the finite dimensional algebra $\End_{\mathsf D(A)}(M)$ is $\tau$-tilting finite.
\end{enumerate}
\end{theorem}

A few words on the last condition are in order. 
In \cite[Proposition 3.27]{AMY} it is shown that for any basic silting complex $M$ with endomorphism ring $E=\End_{\mathsf D(A)}(M)$, there is a bijection between the subset $2-\mathsf{silt}_M\,\T$ of  $\mathsf{silt}\,\T$ consisting of the objects $T$ such that  $M\ge T\ge M[1]$ and the functorially finite torsion classes in $\modr E$. Keeping in mind  that the latter are in bijection with isomorphism classes of basic support $\tau$-tilting  $E$-modules (Theorem~\ref{AIRbij}), one obtains the equivalence of conditions (4) and (5).
Now one can use 
 Theorem~\ref{DIJ} which states that  a finite dimensional algebra is $\tau$-tilting finite if and only if all silting modules are finite dimensional up to equivalence. In fact, this allows to prove the following result.
  
 \begin{theorem}\cite{APV}
A finite dimensional algebra is silting discrete if and only if  all silting complexes are compact up to equivalence.
\end{theorem}

The classification of silting complexes over silting-discrete algebras thus reduces to compact complexes. Moreover,  the latter can all be obtained by iterated silting mutation from the silting object $A$.

\begin{theorem}\cite[Corollary 3.9]{Aihara} If $A$ is a silting-discrete algebra, then the Hasse quiver of the poset $(\mathsf{silt}\,\T,\ge)$ is connected.\end{theorem}

\begin{examples} 
(1) Every  local algebra is silting-discrete since  $\mathsf{silt}\,\T=\{A[i] \mid i\in\mathbb Z\}$, see \cite[Theorem 2.26]{AI}.

(2) A finite-dimensional hereditary algebra is silting-discrete if and only if it is $\tau$-tilting-finite, or equivalently, it has  finite representation type.

(3) Every symmetric algebra of finite representation type is silting-discrete \cite{Aihara}. 

(4) A finite-dimensional algebra $A$ is said to be \emph{derived-discrete} if for every map $v:\mathbb Z\to K_0(\mathsf{D}^b(\modr A))$ 
there are only finitely many isomorphism classes of objects $D$ in $\mathsf{D}^b(\modr A)$ such that  the equality $[H^i(D)]=v(i)$ holds in  $K_0(\mathsf{D}^b(\modr A))$ for all $i\in\mathbb Z$.
 In \cite[Theorem 6.12]{BPP2}, it is proven that all derived-discrete algebras of finite global dimension are silting-discrete.

(5) Every  preprojective algebra of Dynkin type is silting-discrete \cite{AM}.  

(6) Algebras of dihedral, semidihedral or quaternion type are shown to be silting-discrete in \cite{Eisele}. These are the algebras occurring in Erdmann's classification of tame blocks of group algebras of finite groups. 

(7) The Brauer graph algebras that 
are tilting-discrete are determined in \cite{AAChan}.
\end{examples}

Many of the examples above come along with 
  classification results. For example,
  employing results from \cite{Mizuno}, it is shown in \cite{AM}  that the two-term silting complexes over a  preprojective algebra of Dynkin type $A$ are in bijection with the elements of the Weyl group of the underlying Dynkin quiver   $\Delta$. This is used to  
prove  that $A$  is silting-discrete and to give a complete classication of the tilting complexes in $\mathsf{D}(A)$ in terms of the braid group of the folded graph of $\Delta$. 
  
Moreover, \cite{Eisele} contains a  combinatorial description of support-$\tau$-tilting modules over string algebras, which leads,   via a reduction theorem, to a classification of all two-term tilting complexes over  algebras of dihedral, semidihedral and quaternion type. 
 The case of  
 a Brauer graph algebra is treated in \cite{AAChan},
for Brauer tree algebras see also \cite{Zv}.  Notice that here compact silting and tilting complexes coincide as the algebras are symmetric, cf.~\cite[Example 2.8]{AI}.

A classification of all silting complexes over 
derived-discrete algebras of finite global dimension is given in \cite{BPP1}. Over such algebras, the heart of every  bounded t-structure is  a length category. According to Theorem~\ref{thm:ky-bijection}, one thus obtains also a classification of all bounded co-t-structures and t-structures in $\T$.
This is used in \cite{BPP2} to 
provide a combinatorial approach to the study of Bridgeland's stability conditions. As a consequence, it is shown that the stability manifold of a derived-discrete algebra of finite global dimension is always contractible. More recently, this result has been extended to the class of all silting-discrete algebras in \cite{PSZ}, and independently, in \cite{AMY}. Further results relating stability conditions with silting theory can be found in \cite{BST}.

\smallskip

Let us close the paper with a further interesting feature of silting-discrete algebras.
\begin{definition} An object $T$ in a triangulated category $\T$ is said to be \emph{presilting} if $T$ lies in $T^{\perp_{>0}}$. If, in addition,  the category $T^{\perp_{>0}}$ is closed under all existing coproducts, then $T$ is said to be \emph{partial silting}.
\end{definition}

\begin{theorem}\cite[Theorem 2.15]{AM} If $A$ is a silting-discrete algebra, then any presilting object $T$ in $\T=\mathsf{K}^b(\mathrm{proj}(A))$ admits a complement $T'$ such that $T\oplus T'$ belongs to $\mathsf{silt}\,\T$.
\end{theorem}
The existence of complements plays an important role in many contexts, see e.g.~\cite[Corollary 0.7]{AIR}, \cite[Theorem 3.12]{AMV1}, \cite[Section 5]{BPP2}. The case of a large presilting object will be discussed in the forthcoming paper \cite{APV}. The paper \cite{AMV5} will be devoted to partial silting objects. It will be shown that in the derived category $\mathsf D(A)$ of an arbitrary ring $A$ every set $\Sigma\subset\mathsf{K}^b(\mathrm{proj}(A))$ of compact objects admits a partial silting object $T$ such that $\Sigma^{\perp_{\mathbb Z}}=T^{\perp_{\mathbb Z}}$, in analogy with a result for two-term complexes from  \cite{MS}.


\end{document}